\documentclass[12pt,reqno]{amsart}

\usepackage{amsmath}
\PassOptionsToPackage{hyphens}{url}\usepackage[hidelinks,hypertexnames=false]{hyperref}

\usepackage[utf8]{inputenc}
\usepackage[british]{babel}
\usepackage[noabbrev, nameinlink]{cleveref}
\usepackage{amsthm}
\usepackage{amssymb}
\usepackage{mathrsfs}
\usepackage{indentfirst}
\usepackage{stmaryrd}
\usepackage{bbm}
\usepackage{xspace, xcolor}
\usepackage{thmtools}
\usepackage{thm-restate}
\usepackage{calc}
\usepackage{fullpage}
\usepackage{comment}
\usepackage{csquotes}

\usepackage[backend=biber, style=ext-numeric, natbib=true, safeinputenc=true,
  maxbibnames=99, giveninits=true,
url=false, isbn=false, doi=false, articlein=false, mincitenames=99, maxcitenames=100, sorting=nyt]{biblatex}
\DeclareFieldFormat{pages}{#1}
\DeclareFieldFormat[article]{title}{#1\isdot}

\makeatletter
\def\SL@eqntext#1{\rlap{\hspace{-2cm}{\showlabelsetlabel{\tiny \color{blue}\SL@prlabelname{#1}}}}}
\makeatother

\makeatletter
\newcommand{\IfRestatedTF}[2]{\ifthmt@thisistheone #1\else #2\fi}
\makeatother

\newcommand*{\ie}{i.e.\@\xspace}

\newcommand*{\Freiman}{Fre\u{\i}man\xspace}
\usepackage[shortlabels]{enumitem}

\newcommand{\pfr}{\textup{pfr}}

\newcommand{\R}{\mathbb{R}}
\newcommand{\Z}{\mathbb{Z}}
\newcommand{\ZZ}{\mathbb{Z}}
\newcommand{\N}{\mathbb{N}}
\newcommand{\NN}{\mathbb{N}}
\newcommand{\FF}{\mathbb{F}}
\newcommand{\QQ}{\mathbb{Q}}
\newcommand{\RR}{\mathbb{R}}

\newcommand{\eps}{\varepsilon}
\newcommand{\doubling}{\sigma}

\newcommand{\cX}{\mathcal{X}}
\newcommand{\containers}{\mathcal{X}}
\newcommand{\containersHost}{\Sigma}

\newcommand{\cF}{\mathcal{F}}
\newcommand{\calF}{\mathcal{F}}

\newcommand{\cS}{\mathcal{S}}

\newcommand{\cH}{\mathcal{H}}

\newcommand{\dimF}{\dim_\mathrm{F}}
\newcommand{\rank}{\operatorname{rank}}

\newcommand{\conv}{\operatorname{conv}}

\newcommand{\norm}[1]{\lVert#1\rVert}

\newcommand{\equivc}[1]{\llbracket#1\rrbracket}
\newcommand{\oN}[1]{N(#1)}

\crefname{prop}{proposition}{propositions}
\crefname{obs}{observation}{observations}
\Crefname{prop}{Proposition}{Propositions}
\Crefname{obs}{Observation}{Observations}

\title{Counting sets with given doubling via dimension}

\author{Marcelo Campos \and Gabriel Dahia \and Jo\~ao Pedro Marciano}

\address{IMPA, Estrada Dona Castorina 110, Jardim Bot\^anico, Rio de Janeiro, 22460-320, Brasil}
\email{\{marcelo.campos, joao.marciano\}@impa.br}

\address{Institute of Mathematics, EPFL, Lausanne, Switzerland}
\email{gabriel.dahiafernandes@epfl.ch}

\usepackage{graphicx}

\newtheorem{thm}{Theorem}[section]
\newtheorem{lem}[thm]{Lemma}
\newtheorem{defi}[thm]{Definition}
\newtheorem{cor}{Corollary}[thm]

\newtheorem{prop}[thm]{Proposition}

\newtheorem{obs}[thm]{Observation}

\newtheoremstyle{named}{}{}{\itshape}{}{\bfseries}{.}{.5em}{\thmnote{#3}}
\theoremstyle{named}

\begin{document}

\begin{abstract}
  We determine, up to a factor of $2^{o(k)}$, the number of $k$-sets $A \subset \{1, \ldots, n\}$ such that $|A + A| \leq m$, where $k = \Theta(\log n)$ and $m \le k^{1 + \alpha}$, for small $\alpha > 0$, answering a question of \citeauthor{green2014counting}.
\end{abstract}

\maketitle

\section{Introduction}

In his seminal work on random Cayley graphs, \citet{green2005counting} observed that counting sets in the collection
\begin{equation*}
  \Lambda_{n, k, m} = \big\{ A \subset [n] : |A| = k \, \text{ and } \, |A + A| \leq m \big\}
\end{equation*}
is sufficient to estimate the clique number of the uniform random Cayley sum graph defined over $\ZZ_n$.
\citet{green2014counting} improved \citeauthor{green2005counting}'s count of sets in $\Lambda_{n, k, m}$ and determined that, with high probability, the clique number of this graph is asymptotically equal to that of $G(n, 1/2)$.
They proved, in particular, that when $m=O(k)$,
\begin{equation}\label{eq:green-morris}
  |\Lambda_{n, k, m}|  \leq 2^{o(k)} \binom{m/2}{k} n^{\lfloor m/k  + o(1) \rfloor}.
\end{equation}
Note that taking all $k$-subsets of an arithmetic progression of size $m/2$ gives $\binom{m/2}{k}$ choices for the sets in $\Lambda_{n, k, m}$, which is the dominant term in \eqref{eq:green-morris} when $k/\log n \to \infty$.
On the other hand, the union of $t = \lfloor m/k - 2 + o(1) \rfloor$ random points with an arithmetic progression of size $k - t$ gives $n^{\lfloor m/k + o(1) \rfloor}$ choices for the elements in $\Lambda_{n, k, m}$, which is the dominant term in \eqref{eq:green-morris} when $k/\log n \to 0$.
Together with \eqref{eq:green-morris}, these two constructions determine $|\Lambda_{n, k, m}|$ up to a factor of $2^{o(k)} n^{o(1)}$, except when $k = \Theta(\log n)$.

Motivated by this, \citeauthor{green2014counting} asked if one could determine the size of $\Lambda_{n, k, m}$ with the same precision in this intermediate regime.
Our main result is a complete answer to this question, proving bounds on the number of elements of $\Lambda_{n, k, m}$ that are tight up to an error term of the form $2^{o(k)}$.
In fact, we also strengthen their original counting result for all $k$: instead of requiring $m=O(k)$, we determine the number of $k$-sets $A \subset [n]$ with $|A + A| \leq m$ for all $m \leq k^{1 + \alpha}$, for some small constant $\alpha>0$.

To motivate the form of the bound we prove, consider the following construction, which generalises the preceding examples and first appeared implicitly in~\cite{green2014counting} and explicitly in~\cite{campos2020number}.
First, fix some positive integer $d$ and choose an arithmetic progression $P \subset [n]$ of size
\begin{equation}\label{eq:def-of-lambda}
  \lambda_{k,m}(d) = \bigg\lfloor\frac{m-(d-1)k + \binom{d+1}{2}}{2}\bigg\rfloor-d+1.
\end{equation}
We then form $A$ by taking the union of a subset of $P$ of size $k-d+1$, with $d-1$ random elements of $[n]\setminus P$.
There are approximately $n^2$ choices for $P$, $\binom{|P|}{k-d+1}$ choices for the subset of $P$, and $n^{d-1}$ choices for the remaining elements, which results in at least
\[n^{d+1} \binom{\lambda_{k,m}(d)}{k-d+1}\]
sets.
We will show in \Cref{sec:construction} that there exists a large collection of such sets that are distinct, and also that every set obtained in this way belongs to $\Lambda_{n,k,m}$.

We now state our main theorem, which shows that this construction gives the correct order of magnitude for $|\Lambda_{n, k, m}|$, up to a factor of $2^{o(k)}$.

\begin{thm}\label{thm:main}
  There exists an absolute constant $\alpha > 0$ such that the following holds.
  For all $n, k, m \in \N$ satisfying $m \leq n/4$ and $m \leq k^{1 + \alpha}$, we have that
  \begin{equation}\label{eq:main}
    |\Lambda_{n, k, m}| = 2^{o(k)} \sum_{d}\, n^{d + 1} \, \binom{\lambda_{k,m}(d)}{k - d +1},
  \end{equation}
  where the sum is over all $d \in [k-1]$ satisfying
  \begin{equation}\label{eq:condition-on-d-main}
    (d + 1)k - \binom{d+1}{2} \le m.
  \end{equation}
\end{thm}

If $k = C \log n$ for some $C>0$, then the values of $d$ that dominate the sum vary depending on the value of $C$.
For larger values of $C$, we are precisely in the regime where the binomial term dominates in \eqref{eq:green-morris}; in this case, the dominant contribution in \eqref{eq:main} comes from $d=1$.
On the other hand, when $C$ is very small, the values of $d$ close to the maximum satisfying \eqref{eq:condition-on-d-main} dominate the right-hand side of \eqref{eq:main}, namely $d = \lfloor m/k - 1 + o(1)\rfloor$.
This regime of $k$ corresponds to when the term $n^{\lfloor m/k + o(1)\rfloor}$ is the dominant contribution in \eqref{eq:green-morris}.

The main technical tool we need to prove \Cref{thm:main} is a novel version of \Freiman's lemma via few translates.
\Freiman's lemma states that, whenever a set $X\subset \RR^d$ has rank $r$, then we have the following lower bound on the size of its sumset:
\[|X+X| \geq (r+1)|X|-\binom{r+1}{2}.\]
The statement below strengthens this result by showing that one can obtain a similar lower bound on $|X + T|$ for some small subset $T \subset X$.

\begin{restatable}{thm}{fewTranslates}\label{thm:few-translates}
  There exists a constant $C > 0$ such that the following holds.
  Let $d, r \in \N$ and let $\gamma > C \log (2r) / r$.
  If $X \subset \Z^d$ is a finite set with $\rank(X) \geq r$ and $|X|\geq r/(2\gamma)$, then there exists $T \subset X$ such that
  \begin{equation}\label{eq:fewTranslatesBound}
    |X + T| \geq (1 - \gamma)(r + 1)|X|
  \end{equation}
  and $|T| \le C r^C$.
\end{restatable}

The above result contributes to a recent line of work seeking lower bounds for sumsets without using the full set in the sumset.
The study of these questions was initiated in \citep{bollobas2022large}, and has since been continued by several authors \citep{fox2022small,jing2023kemperman}, including in our previous work on the independence number of random Cayley sum graphs \citep{campos2024independence} (see \Cref{sec:overview-freimans-lemma} for a detailed discussion of prior results).
\Cref{thm:few-translates} improves the lower bound on $|X+T|$ that we proved in \citep[Theorem 1.3]{campos2024independence} by a factor of $2$, while still keeping $|T|$ a polynomial function of $r$.
The proof of \Cref{thm:few-translates} uses the machinery developed in \citep{campos2024independence} to prove a ``weighted'' version of \Freiman's lemma, but replaces a simple greedy geometric argument used there by an application of the ``weak'' polynomial \Freiman--Ruzsa theorem of \citet*[Theorem 1.3]{gowers2023conjecture}.

In the proof of our main theorem, \Cref{thm:main}, we use \Cref{thm:few-translates} together with an application of the container method for sumsets, introduced by \citet{campos2020number}.
This method provides, in our context, a special collection $\cX$ of at most $2^{o(k)}$ sets with the property that every set $A$ under consideration is contained in some ``container'' $X \in \cX$.
We can therefore count these sets by first choosing $X \in \cX$, with $2^{o(k)}$ choices, and then choosing $A \in \binom{X}{k}$.
Therefore, to obtain the upper bound in \Cref{thm:main}, it suffices to bound $|X|$ for each $X \in \cX$.
This is the next step in our argument, which we refer to as the supersaturation result, and where \Freiman's lemma via few translates plays a crucial role.
Most of the complications appear at this point: the result that we require generalises previous supersaturation results in \cite{campos2020number, campos2024independence, liu2024number} by relying on additional geometric information about these sets.

One important aspect of our application of the container method for sumsets is that it requires a small host set.
Therefore, we start the proof of \Cref{thm:main} by reducing the problem of counting sets in $\Lambda_{n, k, m}$ to that of counting sets in another, more geometric collection, which we denote by $\Lambda^{(d)}_{k, m}$.
This is a family of $k$-subsets $A$ of a box $[-s,s]^d \cap \ZZ^d$, where $s$ is small and given by an application of a variant of \Freiman's theorem.
In addition to satisfying $|A + A| \le m$, these sets are also full rank, that is, they are not contained in any proper affine subspace of $\RR^d$.
Reducing the counting problem to a smaller, more structured ambient set (e.g.\ $\ZZ_p$, APs, GAPs) is a well-established approach that has appeared in several settings \citep{green2014counting, campos2020number, campos2024independence, liu2024number}.
However, the variant of \Freiman's theorem that we use in this work is different, and its geometric aspect plays a crucial role in the proof.

In the next section, we provide a complete overview of the proof of the upper bound in \Cref{thm:main}.
We explain how we overcome the obstacles that had hindered previous approaches, and we describe the structure of the paper and how the different sections fit together to establish \Cref{thm:main}.

\section{Overview of the proof}\label{sec:overview}

In this section we outline the proof of the upper bound in \Cref{thm:main}.
The setup is the same as in its statement: we want to determine the size of $\Lambda_{n, k, m}$, that is, to count sets $A \subset [n]$ such that $|A| = k$ and $|A + A| \le m$.
To simplify the notation, we will use throughout $\doubling[A] = |A + A| / |A|$ to denote the doubling of $A$, and, when the set is clear from context (or only an upper bound on the doubling is required) we will abbreviate this to $\doubling$.
For example, unless noted otherwise, we have $\doubling = m / k$.
With this notation, the hypothesis in \Cref{thm:main} is exactly $\sigma \leq k^{\alpha}$ for some small constant $\alpha > 0$.

Most of the proof of the upper bound in \Cref{thm:main} takes place in $\ZZ^d$ instead of $[n]$.
Our first step is therefore to change the ambient setting, using a version of \Freiman's theorem~\cite{freiman1973foundations, ruzsa1994generalized} (see \Cref{thm:fullrank-freimans-thm}).
The rank of a set $B \subset \RR^d$, denoted by $\rank(B)$, is the dimension of the smallest affine subspace of $\RR^d$ that contains $B$.
Moreover, a set $B \subset \RR^d$ is called \emph{full rank} if it is not contained in any proper affine subspace of $\RR^d$, \ie, $\rank(B)=d$.

\begin{defi}\label{def:def-of-Lambda-d}
  Consider the box $\Sigma_d := [-s, s]^d \cap \ZZ^d$, where
  \begin{equation}\label{eq:value-of-s-overview}
    s =  \exp\big(O(\sigma\log \sigma)^3\big) k^{2 \doubling}.
  \end{equation}
  Define $\Lambda^{(d)}_{k, m}$ to be the family of full rank sets $B \subset \Sigma_d$ satisfying $|B| = k$ and $|B + B| \le m$.
\end{defi}

Our version of \Freiman's theorem, \Cref{thm:fullrank-freimans-thm}, states, in essence, that every $A \in \Lambda_{n, k, m}$ is \Freiman isomorphic to some $B \in \Lambda^{(d)}_{k, m}$, where $d$ is at most the \Freiman dimension of $A$, denoted by $\dimF(A)$ (see \Cref{sec:main} for the definitions of \Freiman isomorphism and \Freiman dimension).
Furthermore, the resulting \Freiman isomorphism is affine. This allows us to count such maps using elementary linear algebra (see \Cref{prop:counting-affine-maps}).

An important fact about \Freiman dimension that we will use repeatedly is the following result, known as \Freiman's lemma.
\begin{lem}[{\cite[Lemma~1.14]{freiman1973foundations}}]\label{stmt:freimansLemma}
  For any \(A \subset [n]\) we have
  \[|A+A| \geq (d + 1)|A|- \binom{d + 1}{2}\]
  where $d = \dimF(A)$.
\end{lem}

Together, these results yield
\begin{equation}\label{eq:reducing-Lambda-to-Zd}
  |\Lambda_{n, k, m}| \le \sum_{d} n^{d + 1} \, \big|\Lambda^{(d)}_{k, m}\big|
\end{equation}
under the hypotheses of \Cref{thm:main}, where (by \Cref{stmt:freimansLemma} and since $d \leq \dimF(A)$) the sum ranges over those $d \in [k-1]$ satisfying \eqref{eq:condition-on-d-main}.

We have thus reduced our problem to an upper bound for the size of $\Lambda^{(d)}_{k, m}$, which will be the focus of the remainder of this \namecref{sec:overview}.
More precisely, it is enough to prove the following lemma, which, combined with \eqref{eq:reducing-Lambda-to-Zd}, implies the upper bound in \Cref{thm:main}.

\begin{restatable}{lem}{RedToZd}\label{lem:reduction-to-Zd}
  There exists an absolute constant $\alpha > 0$ such that the following holds.
  If $k, m \in \N$ satisfy $m \leq k^{1 + \alpha}$, then
  \begin{equation*}
    \big|\Lambda^{(d)}_{k,m}\big| \leq 2^{o(k)}\binom{\lambda_{k,m}(d)}{k-d+1}
  \end{equation*}
  for any $d \in [k-1]$ satisfying
  \begin{equation*}
    (d + 1)k - \binom{d+1}{2} \le m.
  \end{equation*}
\end{restatable}

Our approach to proving \Cref{lem:reduction-to-Zd} is to use a container theorem for sumsets, proved by \citet{campos2020number} via the asymmetric container lemma of \citet{morris2024asymmetric}.
This result, presented here as \Cref{thm:containers-sumset}, produces a small family of pairs of sets, called container-pairs, such that for every $B \in \Lambda^{(d)}_{k, m}$ there is a container-pair $(X, Y)$ with $B \subset X$ and $Y \subset B + B$.
Moreover, either $X$ is small or the set $Y$ approximates the sumset $B + B$ according to a notion that we will introduce later.
We will combine \citeauthor{campos2020number}'s container theorem for sumsets with a supersaturation result (\Cref{thm:supsat}), effectively converting this notion of $Y$ approximating $X+X$ into an upper bound on the size of $X$.

Consequently, we obtain a collection $\containers$ of subsets of $\Sigma_d$ with
\begin{equation}\label{eq:size-of-our-container-family-overview}
  |\containers| = 2^{o(k)}
\end{equation}
such that:
\begin{enumerate}[(a)]
  \item For every $B \in \Lambda^{(d)}_{k, m}$, there is $X \in \containers$ such that $B \subset X$.
    \label{item:containers-contain-overview}
    \smallskip
  \item For all $X \in \containers$,
    \begin{equation}\label{eq:containers-are-small-overview}
      |X| \le \big(1 + o(1)\big) \, \max_{r \in [d]} \, \frac{m - (d - r)k}{r+1}.
    \end{equation}\label{item:containers-are-small-overview}
\end{enumerate}
\vspace{-\smallskipamount}
See \Cref{thm:containers} for the precise statement of this result.
It is essential that we are in the smaller host set $\Sigma_d=[-s,s]^d \cap \ZZ^d$ rather than in $[n]$ for two reasons.
The first, and more straightforward, is that the size of the container family $\containers$ depends on the size of $\Sigma_d$.
The second, more subtle, reason is that applying the containers for sumsets theorem of \citep{campos2020number} (\Cref{thm:containers-sumset}) to prove our container result (\Cref{thm:containers}) requires
\begin{equation}\label{eq:upperBoundOnM}
  m \le k^2/(\log |\Sigma_d|)^3,
\end{equation}
which follows from \eqref{eq:value-of-s-overview}, our bound on $s$.
Recently,  \citet{liu2024number} improved the required upper bound on $m$ for the container statement.
However, even replacing the right-hand side of \eqref{eq:upperBoundOnM} by $k^2/(\log |\Sigma_d|)^{2}$ still requires an upper bound on $|\Sigma_d|$, so using their results would not simplify our proof.

Another important feature of working in $\ZZ^d$ is that rank is monotone under inclusion.
Thus, if $B \subset X \subset \ZZ^d$ and $\rank(B) = d$, then $\rank(X) = d$ as well.
Consequently, we may assume that the containers, like $B$, have full rank.
This simple observation will be important when we apply the supersaturation result to bound the size of a container $X$.

Using the container family $\cX$ and its properties, we can explain how we count the sets $B \in \Lambda^{(d)}_{k, m}$.
First, notice that, by the first property of $\cX$ (\cref{item:containers-contain-overview} above), it suffices to choose a container $X \in \containers$ and then choose $B \subset X$.
As our container family is small, by \eqref{eq:size-of-our-container-family-overview}, the choices of $X$ contribute only a factor of $2^{o(k)}$.
For each fixed container $X$, then, there are $\binom{|X|}{k}$ choices for $B$.
Combining the choices for $X \in \cX$ and $B \subset X$ then yields
\begin{equation*}
  \big|\Lambda^{(d)}_{k, m}\big| \le \sum_{X \in \containers} \binom{|X|}{k} \le 2^{o(k)} \, \max_{r \in [d]}\, \binom{\frac{m - (d - r)k}{r+1}}{k} \le 2^{o(k)} \binom{\lambda_{k,m}(d)}{k-d+1},
\end{equation*}
where the second inequality follows from the second property of the container family, \cref{item:containers-are-small-overview}, and the final inequality follows from standard binomial estimates (see Appendix \ref{sec:appendix-calc}).
This finishes the outline of the proof of \Cref{lem:reduction-to-Zd}, and hence of \Cref{thm:main}, up to proving our container theorem.

We now discuss how to derive the upper bound in \cref{item:containers-are-small-overview} from the container theorem for sumsets of \citet{campos2020number}, \Cref{thm:containers-sumset}.
To deduce that each container satisfies \eqref{eq:containers-are-small-overview}, \ie
\begin{equation}\label{eq:containers-are-small-repeated}
  |X| \le \big(1 + o(1)\big) \, \max_{r \in [d]} \,  \frac{m - (d - r)k}{r+1},
\end{equation}
we use the fact that $Y \subset B + B$ approximates $X + X$ in the following sense:
\begin{equation}\label{eq:few-missing-pairs-overview}
  |\{ (x_1, x_2) \in X^2 : x_1 + x_2 \notin Y \}| \le \eps |X|^{2},
\end{equation}
for some small and carefully chosen $\eps > 0$ depending on $k$,  $d$ and $\sigma$.

The proof of \eqref{eq:containers-are-small-repeated} is somewhat intricate.
We therefore focus here on one of its main ideas: deriving a strong lower bound for $|Y|$ from the approximation in \eqref{eq:few-missing-pairs-overview}.
To this end, we say that a set $X$ has $\delta$-robust rank $r$ if every subset $X' \subseteq X$ with $|X'| \ge (1 - \delta)|X|$ satisfies $\rank(X') \geq r$.

\begin{restatable}{thm}{supsat}\label{thm:supsat}
  There exists an absolute constant $C>0$, and for every $0<\gamma<1/2$ there exists $c=c(\gamma)>0$, such that for all $d,r\in\NN$ and $0<\delta<\gamma$, the following holds.
  For all finite sets $X \subset \ZZ^d$ with $|X| \ge 40 (r + 1)^2 / \gamma$ and $\delta$-robust rank $r$, if $Y \subset X + X$ satisfies
  \begin{equation*}
    |\{ (x_1, x_2) \in X^2 : x_1 + x_2 \notin Y \}| \le c \delta r^{-C}|X|^{2},
  \end{equation*}
  then
  \begin{equation}\label{eq:supsat-large-y}
    |Y| \geq (1 - \gamma) (r + 1) |X|.
  \end{equation}
\end{restatable}

While the container $X$ need not have $\delta$-robust rank $r$ for any $r \in [d]$, a simple greedy argument (see \Cref{lem:robustify}) guarantees the existence of a large robust subset of $X$.
Moreover, since $Y \subset B + B$, we trivially have $|Y| \le m$, which, if combined with \eqref{eq:supsat-large-y}, would yield something like $|X| \le m / (r + 1)$.
To deduce \eqref{eq:containers-are-small-repeated} instead, \ie with the correction term $(d - r)k$, in \Cref{lem:size-of-Y} we rely on the entirety of $X$, and not only on its large robust subset.

We note that our recent work on random sparse Cayley graphs contains a supersaturation result \citep[Theorem~2.2]{campos2024independence} that is closely related to \Cref{thm:supsat}.
In fact, we derive \Cref{thm:supsat} from \Cref{thm:few-translates} by adapting the ``few-translates implies supersaturation'' strategy introduced in \citep{campos2024independence} and subsequently used in \citep{nenadov2025remark}.
In addition to imposing a stronger condition on the number of missing pairs, our result gives a lower bound for $|Y|$ in \eqref{eq:supsat-large-y} that is twice as large as the corresponding bound in \cite{campos2024independence}.
This improvement follows from the strengthened version of \Freiman's lemma via few translates \Cref{thm:few-translates}, which we discuss next.

\subsection{Improved \Freiman's lemma via few translates}\label{sec:overview-freimans-lemma}

The main technical input in the proof of \Cref{thm:supsat} is \Cref{thm:few-translates}, restated below for convenience.
It is an approximate version of \Freiman's lemma, \Cref{stmt:freimansLemma}, that replaces one of the sets in the sumset by a subset with few translates.

\fewTranslates*

The above statement fits into a recent line of work on proving lower bounds for sumsets without using the full set in the sumset.
It was initiated by \citet{bollobas2022large}, who showed that, for any $A, B \subset \ZZ_p$ with $|A| = |B| \le p/3$, there is a constant-sized subset $B' \subset B$ such that
\begin{equation}\label{eq:bltModP}
  |A + B'| \geq |A| + |B| - 1.
\end{equation}
Their proof first established that, when $A, B \subset \Z$ and $|A| \ge |B|$, a 3-subset $B' \subset B$ suffices for \eqref{eq:bltModP}, see \Cref{thm:blt}, and they moreover conjectured that the same would hold in $\ZZ_p$.
This conjecture was subsequently proved by \citet{fox2022small}, whose argument uses, among other ingredients, a few-translates lower bound with a polynomial dependence on the doubling $\sigma$.
The key feature of this result in our setting (see \Cref{thm:flpz}) is that, whenever $r = \rank(X) = O(\sigma^{1/3})$, it recovers the bound in \eqref{eq:fewTranslatesBound} for $|X+T|$ for some $T \subset X$ with $|T| = O(r)$.

Even closer to \Cref{thm:few-translates} is a result of \citet{jing2023kemperman}.
They obtained a sharp few-translates version of \Freiman's lemma if the rank $r$ is fixed.
That is, they proved (see \Cref{thm:jing-mudgal}) that for any finite set $X \subset \Z^d$ with $\rank(X) = r$, there exists $T \subset X$ such that
\[
  |X + T| \geq (r+1)|X| - O(r^3)
\]
with $|T| \le C_r$ for some constant $C_r>0$ depending on $r$.
The dependence of $C_r$ on $r$ is unfortunately super-exponential, which is insufficient for our current application.
In fact, while none of these results alone yield \Cref{thm:few-translates}, we will use all of them as part of our argument in \Cref{sec:supsat}.

In \cite[Theorem~1.3]{campos2024independence}, we proved a variant of \Cref{thm:few-translates} with $|T| = O_\gamma(r)$ that obtains only
\[
  |X + T| \ge (1 - \gamma) \frac{r + 1}{2} |X|,
\]
\ie it loses a factor of $1/2$ compared to \eqref{eq:fewTranslatesBound}.
This constant factor loss is due to the first step in the proof, a greedy geometric argument that either finds a set of translates with the desired properties, or determines that a large portion of $X$ lies in a lower-dimensional subspace.
As the missing factor of $2$ is essential for our current application, and we could not improve this simple greedy argument to obtain \eqref{eq:fewTranslatesBound}, we relied on a different strategy to establish \Cref{thm:few-translates}.

Our approach to overcome that issue is to separate the problem in two cases, depending on $\doubling[X]$, the doubling of the set $X$.
Consider first when $\doubling[X] = O(r^3)$, which is the most interesting case.
In this case, we use this bound on the doubling to replace the simple greedy argument and directly determine that there is a large subset of $X$ in a subspace of low dimension.
To find such a subset, we rely on the ``weak'' Polynomial \Freiman--Ruzsa theorem of \citet*{gowers2023conjecture}, a consequence of their proof of Marton's conjecture in characteristic 2 (also known as the PFR conjecture in $\FF_2^n$).
This theorem (which was shown in \cite{green2025sumsets} to follow from PFR in $\FF_2^n$) says that every finite $X \subset \ZZ^d$ with doubling at most $\doubling$ has a subset $X_0 \subset X$ such that $|X_0| \ge \doubling^{-C} |X|$ and $\rank(X_0) \le C \log \doubling$, for some fixed constant $C > 0$ (see \Cref{thm:weak-pfr}).
The role of this large, low-dimensional set $X_0$ is to
provide a structured base over which we decompose $X$ into ``fibres''.
Once this is achieved, we can then apply the (intricate) machinery of the ``weighted'' version of \Freiman's lemma developed in \cite{campos2024independence} to these fibres, with the weights recording their sizes.

The other case, when $\doubling[X] \geq C r^3$ for some absolute constant $C>0$, follows immediately from the result of \citet*[Theorem 1.1]{fox2022small}, which provides a set $T \subset X$ such that $|T| = O(r^3)$ and $|X + T| \ge (r + 1)|X|$.

\subsection{Organisation of the paper}
In \Cref{sec:main}, we prove the upper bound in \Cref{thm:main} under the assumption that our container theorem, \Cref{thm:containers},  and our version of \Freiman's theorem, \Cref{thm:fullrank-freimans-thm}, hold.
\Cref{sec:containers} is dedicated to prove \Cref{thm:containers} using the container theorem for sets with small sumset, \Cref{thm:containers-sumset}, and the supersaturation result, \Cref{thm:supsat}, whose proof is given in \Cref{sec:supsat} together with the proof of \Cref{thm:few-translates}.
\Cref{sec:fullrank-freimans-thm} contains the proof of the full rank version of \Freiman's theorem, and the final \namecref{sec:construction}, \Cref{sec:construction}, completes the proof of \Cref{thm:main} with the corresponding lower bound.

Throughout the paper, we omit floors and ceilings in our calculations, except when doing so would be likely to cause confusion.

\section{Proof of Theorem~\ref{thm:main}}\label{sec:main}

The goal of this section is to prove the upper bound in \Cref{thm:main} assuming that both our container theorem, \Cref{thm:containers}, and our full rank version of \Freiman's theorem, \Cref{thm:fullrank-freimans-thm}, hold.
That is, denoting $\Lambda=\Lambda_{n,k,m}$ for brevity, we will establish that
\begin{equation*}
  | \Lambda | \leq 2^{o(k)} \sum_{d}\, n^{d + 1}\binom{\lambda_{k,m}(d)}{k - d + 1},
\end{equation*}
where the sum is taken over $d \in [k-1]$ satisfying
\begin{equation}\label{eq:restriction-d}
  (d+1)k-\binom{d+1}{2} \leq m.
\end{equation}
Recall that $\lambda_{k,m}(d)$ is defined in \eqref{eq:def-of-lambda} as
\begin{equation*}
  \lambda_{k,m}(d) = \bigg\lfloor \frac{m - (d-1)k + \binom{d+1}{2}}{2} \bigg\rfloor -d + 1.
\end{equation*}

To start, recall the definition $\Sigma_d=[-s,s]^d \cap \ZZ^d$ for each $d \in \NN$, where
\begin{equation*}
  s = \exp\big(C_0 (\doubling \log \doubling)^3\big)k^{2\doubling}
\end{equation*}
for some constant $C_0$ given by \Cref{thm:fullrank-freimans-thm}.
Further recall that we define $\Lambda^{(d)}_{k,m}$, for each $d \in \NN$, to be the collection of full-rank subsets $B\subset \Sigma_d$ with $|B|=k$ and $|B+B| \leq m$.
As outlined in \Cref{sec:overview}, the first step in our proof of \Cref{thm:main} is to reduce it to proving \Cref{lem:reduction-to-Zd}, restated below for convenience.
This \namecref{lem:reduction-to-Zd} gives an upper bound on the size of $\Lambda^{(d)}_{k,m}$ for each $d \in [k-1]$ satisfying \eqref{eq:restriction-d}.

\RedToZd*

We will show that \Cref{lem:reduction-to-Zd} follows from the container statement described informally in \Cref{sec:overview}, which we now formally state.
Its proof appears in the next \namecref{sec:containers}, \Cref{sec:containers}.

\begin{restatable}{thm}{containersX}\label{thm:containers}
  Let $m, t, d, k \in \NN$ and $\sigma=m/k>0$.
  Let also $0 < \gamma < 1/8$ and assume that the preceding parameters satisfy the conditions
  \begin{equation}\label{eq:conditions-for-containers}
    \sigma^2 \leq \log t \leq \sqrt{m} \qquad \text{and} \qquad k \geq 80(d+1)^2/\gamma.
  \end{equation}
  Finally, let $\containersHost \subset \ZZ^d$ be a set with $|\containersHost| = t$.
  There exists $M=M(\gamma)>0$ depending only on $\gamma$ such that the following holds.
  There is a constant $C >0$ and a family of sets $\containers \subset 2^\containersHost$ of size at most
  \begin{equation*}
    |\containers| \le \exp\Big( M d^{C} \sqrt{\sigma^3 k (\log t)^{3}} \Big)
  \end{equation*}
  such that:
  \begin{enumerate}[(a)]
    \item For all $B \subset \containersHost$ with $|B| = k$, $|B + B| \le m$ and $\rank(B) = d$, there is $X \in \containers$ such that $B \subset X$.
      \label{item:containers-contain}

    \item Every $X \in \containers$ satisfies
      \begin{equation}\label{eq:containers-are-small}
        |X| \le \max_{r \in [d]} \, (1 + 5\gamma) \, \frac{m - (d - r)k}{r+1}.
      \end{equation}
      \label{item:containers-are-small}
  \end{enumerate}
\end{restatable}

Instead of explicitly proving \Cref{lem:reduction-to-Zd}, we will deduce \Cref{thm:main} directly from \Cref{thm:containers}.
To do so, we require some intermediate results and definitions.
Recall that a \Freiman isomorphism $\phi : S \to T$ is a bijection such that for all $s_1, s_2, s_1', s_2' \in S$ with $s_1 + s_2 = s_1' + s_2'$, we have
\begin{equation*}
  \phi(s_1) + \phi(s_2) = \phi(s_1') + \phi(s_2')
\end{equation*}
and vice-versa, that is, both $\phi$ and $\phi^{-1}$ are \Freiman homomorphisms.
We further say that $A$ and $B$ are \Freiman isomorphic if there is a \Freiman isomorphism $\phi : A \to B$.
Recall that we also define the rank of a set $B \subset \ZZ^d$, denoted by $\rank(B)$, to be the smallest dimension $r \in \NN$ of an affine subspace of $\RR^d$ containing $B$, and that we say that $B$ has full rank if $\rank(B) = d$.
The \Freiman dimension of a set $A$, denoted by $\dimF(A)$, is the largest $d \in \NN$ such that there exists a set $B \subset \ZZ^d$ with full rank which is \Freiman isomorphic to $A$.
Finally, we say that a function $\phi : S \subset \RR^d \to \RR$ is affine if there exist $v_0, v_1, \ldots, v_d \in \RR$ such that
\[\phi(s_1,\ldots,s_d) = v_0 + \sum_{i=1}^d s_i v_i\]
for every $(s_1, \ldots, s_d) \in S$.

The first intermediate result that we require is a version of \Freiman's theorem, \Cref{thm:fullrank-freimans-thm} below.
It says that every $A \in \Lambda_{n, k, m}$ is \Freiman isomorphic to a full rank set $B \subset \Sigma_d \subset \ZZ^d$.

\begin{restatable}{thm}{fullrankfreimansthm}
  \label{thm:fullrank-freimans-thm}
  There exists $C_0 > 0$ such that the following holds.
  Let $n, k \in \NN$ and $1 < \doubling \le k$.
  For every $A \subset [n]$ with $|A| = k$ and $|A + A| \le \doubling k$, there exists $B \subset \ZZ^d$, for some $d \in \NN$, and a \Freiman isomorphism $\phi : B \to A$ such that
  \begin{enumerate}[(a)]
    \item $\rank(B) = d$, \label{item:fullrank-freimans-thm}
    \item the function $\phi$ is affine,
    \item $B \subset [-s, s]^d$, where
      \begin{equation}\label{eq:fullrank-freimans-thm-size-of-box}
        s = \exp(C_0 (\doubling \log \doubling)^3) \, k^{2\doubling}.
      \end{equation}
  \end{enumerate}
\end{restatable}

\Cref{thm:fullrank-freimans-thm} follows from combining \citeauthor{chang2002polynomial}'s version of \Freiman's theorem~\cite{chang2002polynomial} with a linear algebraic lemma, due to \citet[Lemma~A.2]{green2022new}, that ensures that $B$ is full rank.
The proof of \Cref{thm:fullrank-freimans-thm} is simple, though somewhat technically tedious, so we defer it to \Cref{sec:fullrank-freimans-thm}.

The final missing ingredients in the proof of \Cref{thm:main} are \Cref{prop:counting-affine-maps}, a simple counting \namecref{prop:counting-affine-maps}, and \Cref{prop:dim-doubling}, a way to bound the \Freiman dimension of a set in terms of its doubling.

\begin{lem}\label{prop:counting-affine-maps}
  Let $n, d\in \NN$ and let $B \subset \ZZ^d$ be a full rank set.
  There are at most $n^{d+1}$ choices for an affine function $\phi : B \to [n]$.
\end{lem}

\begin{proof}
  We will show that $\phi$ is completely determined by choosing $b_0, b_1, \ldots, b_d \in B$ such that $b_1-b_0, \ldots, b_d-b_0$ form a basis for $\RR^d$, and then choosing $\phi(b_0), \ldots, \phi(b_d) \in [n]$.
  As there are at most $n^{d+1}$ choices for $\phi(b_0), \ldots, \phi(b_d) \in [n]$, this will complete the proof.
  Fixing such a choice of $b_0, b_1, \ldots, b_d \in B$, we have, for any $x \in B$, that
  \begin{equation}\label{eq:x-linear-comb}
    x - b_0 = \sum_{i=1}^d \theta_i (b_i - b_0)
  \end{equation}
  for some choice of $\theta_1, \ldots, \theta_d \in \RR$.
  It then follows from \eqref{eq:x-linear-comb} and the fact that $\phi$ is affine that
  \[
    \phi(x) - \phi(b_0) = \sum_{i=1}^d \theta_i \big(\phi(b_i) - \phi(b_0)\big),
  \]
  which determines $\phi(x)$ for every $x \in B$, as we wanted to show.
\end{proof}

\begin{lem}\label{prop:dim-doubling}
  Let $k, m \in \N$ and define $\sigma =m/k>0$.
  For all $d \in [k-1]$, if
  \begin{equation}\label{eq:freimansLemma}
    (d+1)k-\binom{d+1}{2} \le m
  \end{equation}
  then,
  \begin{equation}\label{eq:boundsOnD}
    d \leq 2\sigma-2 \qquad \text{and} \qquad d \leq \sigma-1+\frac{2\sigma^2}{k}.
  \end{equation}
  In particular, if $\sigma \leq \sqrt{k/2}$, then \eqref{eq:boundsOnD} implies that $d \leq \sigma$.
\end{lem}

\begin{proof}
  From the assumption that $k \geq d+1$, we obtain
  \[(d+1)k - \frac{dk}{2} \leq(d+1)k- \binom{d+1}{2} \leq  \sigma k,\]
  which establishes the first inequality in \eqref{eq:boundsOnD}.
  Replacing this improved upper bound in \eqref{eq:freimansLemma},
  \[(d+1)k - \binom{2 \sigma-1}{2} \leq \sigma k\]
  which then implies that
  \[d \leq \sigma -1 + \frac{2\sigma^2}{k}\]
  as we wanted to show.
\end{proof}

We can now prove \Cref{thm:main} from \Cref{thm:containers}.

\begin{proof}[Proof of \Cref{thm:main} assuming \Cref{thm:containers}]
  Note first that we can assume that $k$ is sufficiently large, since the error term is of the form $2^{o(k)}$.
  We first claim that every $A \in \Lambda$ can be obtained by first choosing $d \in \NN$ satisfying
  \begin{equation}\label{eq:freimans-lem-conseq}
    (d+1)k-\binom{d+1}{2} \leq m,
  \end{equation}
  choosing $B \in \Lambda^{(d)}_{k,m}$, and then choosing an affine map $\phi : B \subset \Sigma_d \to A \subset [n]$.
  Fixing $A \in \Lambda$ and applying \Cref{thm:fullrank-freimans-thm}, we obtain $B \subset \ZZ^d$ for some $d$, which is moreover a subset of $\Sigma_d$ with full rank, and that is \Freiman isomorphic to $A$ via an affine map $\phi:B \to A$.
  By the definition of \Freiman dimension and the fact that $B$ is full rank, we conclude that $d \leq \dimF(A) \le k - 1$, so \Freiman's lemma (\Cref{stmt:freimansLemma}) implies that $d$ satisfies \eqref{eq:freimans-lem-conseq}.
  Moreover, note that $|B|=k$ and $|B + B| = |A + A| \leq m$ follows from $\phi$ being affine, and therefore $B \in \Lambda^{(d)}_{k,m}$, establishing our claim.

  With the claim established, we count the number of $A \in \Lambda$ by applying \Cref{prop:counting-affine-maps} to count the choices for the affine map, obtaining as a result
  \begin{equation}\label{eq:first-bound-lambda}
    |\Lambda| \leq \sum_d n^{d+1}|\Lambda^{(d)}_{k,m}|,
  \end{equation}
  where the sum is over $d \in [k-1]$ satisfying \eqref{eq:freimans-lem-conseq}.

  Let $t=|\Sigma_d|=(2s+1)^d$ and fix $k^{-1/2}<\gamma<1/5$ (later in the proof we will take $\gamma$ to be a function of $k$ slowly tending to $0$).
  In order to apply \Cref{thm:containers}, we verify that all the hypotheses in \eqref{eq:conditions-for-containers} hold.
  First, we claim that
  \begin{equation*}
    \sigma^2 \leq \log t \leq \sqrt{m}.
  \end{equation*}
  Note that, by our bound on $s$ in \eqref{eq:fullrank-freimans-thm-size-of-box}, we have
  \begin{equation}\label{eq:log-s}
    \log (2s+1) = 2\sigma\log k+C_0\sigma^3(\log\sigma)^3+O(1) \leq C_1 \sigma^{3}(\log k)^{3},
  \end{equation}
  for a sufficiently large constant $C_1>0$.
  Since $d \leq 2\sigma \leq 2k^\alpha$, \(\alpha < 1 / 8\) and $k$ is sufficiently large, we obtain from \eqref{eq:log-s} that
  \begin{equation*}
    \log t = d \log (2s+1) \leq 2C_1\sigma^{4}( \log k )^{3} \leq 2C_1 k^{4\alpha}( \log k )^{3} \leq \sqrt{k} \leq \sqrt{m}.
  \end{equation*}
  Moreover, the lower bound $\sigma^2 \leq \log t$ follows from the equality in \eqref{eq:log-s} with $d \geq 1$.
  Second, we confirm that
  \[
    k \geq 80(d+1)^2/ \gamma.
  \]
  In fact, since $\gamma >k^{-1/2}$, it is enough to check that
  \[k \geq 80(d+1)^2 k^{1/2}\]
  or
  \begin{equation}\label{eq:k-vs-d-gamma}
    k^{1/2} \geq 80(d+1)^2.
  \end{equation}
  It follows from \Cref{prop:dim-doubling} that $d \leq 2 \sigma \leq 2 k^{\alpha}$, so \eqref{eq:k-vs-d-gamma} holds for sufficiently large $k$ as long as $\alpha < 1/4$.

  Now, apply \Cref{thm:containers} to obtain a small collection $\cX$ as in its statement for some $M=M(\gamma)$ and $C>0$.
  We claim that $\cX$ has size at most $2^{ M k^{3/4}}$; more precisely, recalling that \(d \leq \sigma \leq k^{\alpha}\) and $\log t \leq 2C_1\sigma^4 (\log k)^3$, we have
  \begin{equation}\label{eq:container-family-small}
    |\cX| \leq \exp\bigg(Md^C\sqrt{\sigma^3 k ( \log t )^{3}}\bigg) \leq \exp\bigg(M k^{1/2} \sigma^{C+8}(\log k)^{9/2} \bigg) \leq 2^{M k^{3/4}},
  \end{equation}
  where the last inequality follows from $k$ being sufficiently large and $\sigma \leq k^{\alpha}$ if we take $\alpha < \big(4(C+8)\big)^{-1}$.
  Moreover, for each \(B  \in \Lambda^{(d)}_{k,m}\) there exists \(X \in \cX\) so that \(B \subset X\) and
  \begin{equation*}
    |X| \leq \max_{r \in [d]}\, (1 + 5\gamma)\frac{m - (d - r)k}{r + 1}.
  \end{equation*}

  We bound the number of sets $B \in \Lambda^{(d)}_{k,m}$ from above by summing over all \(X \in \cX\) and counting the number of \(B \subset X\) with $|B|=k$.
  Note that we have at most
  \[\binom{|X|}{k} \leq \max_{r \in [d]}\binom{(1 + 5\gamma)\frac{m - (d - r)k}{r + 1}}{k}\]
  possibilities for the set $A \subset X$ with $|A|=k$.
  So, to finish the proof, all we need are simple calculations to show essentially that the above term is maximised for $r=1$ -- we do these in Appendix \ref{sec:appendix-calc}.
  It follows from \Cref{lem:binom-main} that
  \begin{equation}\label{eq:binomial-container-bound}
    \binom{|X|}{k} \leq \max_{r \in [d]}\binom{(1 + 5\gamma)\frac{m - (d - r)k}{r + 1}}{k} \leq 2^{30\gamma \log(1/ \gamma) k}\binom{\lambda_{k, m}(d)}{k - d + 1}
  \end{equation}
  possibilities for the set $B \subset X$ with $|B|=k$.
  Therefore, combining \eqref{eq:container-family-small} and \eqref{eq:binomial-container-bound}, we obtain
  \[|\Lambda^{(d)}_{k,m}| \leq \sum_{X \in \cX} \binom{|X|}{k} \leq 2^{Mk^{3/4}}2^{30\gamma \log (1/\gamma) k}\binom{\lambda_{k, m}(d)}{k - d + 1}.\]
  Replacing this in \eqref{eq:first-bound-lambda} and taking $\gamma \to 0$ as $k \to \infty$ in a way that ensures that \[M=M(\gamma) = o(k^{1/4})\] completes the proof.
\end{proof}

\section{Containers for sumsets}\label{sec:containers}

In this section, we prove that \Cref{thm:containers} follows from \Cref{thm:supsat}.
For this, we use the following robustification \namecref{lem:robustify}, which follows from the same proof in \citep[Proposition~8.3]{campos2024independence}.

\begin{lem}\label{lem:robustify}
  Let \(d \in {\mathbb{N}}\) and \(X \subset {\mathbb{Z}}^{d}\) be a finite set with full rank.
  For any \(0<\delta < 1/d\), there is \(X' \subset X\), a \(\delta\)-robust subset, such that
  \begin{equation*}
    |X'| \geq (1 - \delta d)|X|.
  \end{equation*}
\end{lem}

Note that, after applying \Cref{lem:robustify}, we can only guarantee that the robust dimension $r$ of $X'$ satisfies $1 \le r \le d = \rank(X)$.
We also need the following version of \citeauthor{campos2020number}' container theorem for sumsets.

\begin{thm}[{\cite[Theorem 4.2]{campos2020number}}]\label{thm:containers-sumset}
  Let $d, t \in \NN$ and $\containersHost \subset \ZZ^d$ with $|\containersHost| = t$.
  Further let $0 < \eps < 1/4$ and $m \in \NN$ satisfy $m \geq (\log t)^2$.
  There exists a family \(\calF \subset 2^\containersHost \times 2^{\containersHost + \containersHost}\) of size at most
  \begin{equation*}
    |\calF| \le \exp\big(2^{16} \eps^{-2} \sqrt{m} ( \log t )^{3 / 2} \big)
  \end{equation*}
  such that:
  \begin{enumerate}[(i)]
    \item For every $B \subset \containersHost$ with $|B + B| \le m$, there is a pair $(X,Y) \in \calF$ such that $B \subset X$ and $Y \subset B + B$.
      \label{item:containers-contain-sets-with-m-sumsets}

    \item For every $(X, Y) \in \calF$, we have that $|Y| \le m$ and moreover that either
      \begin{equation}\label{eq:small-container}
        |X| \le \frac{m}{\log t},
      \end{equation}
      or
      \begin{equation}\label{eq:missing-pairs}
        \big|\{ ( x_1, x_2 ) \in X^2 : x_1 + x_2 \notin Y \} \big| \le \varepsilon^2 |X|^2.
      \end{equation}
  \end{enumerate}
\end{thm}

The final \namecref{lem:size-of-Y} that we need gives bounds on the size of $Y \cap (X' + X')$ as a consequence of \Cref{thm:supsat}.

\begin{lem}\label{lem:size-of-Y}
  Let $k, d, m \in \N$, $0 < \delta \leq 1 / (2d)$ and $0 < \gamma < 1/5$ satisfy
  \begin{equation*}
    m \ge k \geq 80(d+1)^2/\gamma.
  \end{equation*}
  For every pair of finite sets $X,Y \subset \Z^d$ satisfying $|X|,|Y| \leq m$ such that there exists $B \in \Lambda_{k,m}^{(d)}$ satisfying $B \subset X$ and $Y \subset B + B$, the following holds.
  There exist $c=c(\gamma)>0$ and an absolute constant $C>0$ such that if $X' \subset X$ is a $\delta$-robust set satisfying $r=\rank(X')$ and
  \begin{equation}\label{eq:few-missing-pairs-lem}
    \big|\big\{ (x_1, x_2) \in (X')^2 : x_1 + x_2 \notin Y' \big\}\big| \le c \delta r^{-C}|X'|^{2},
  \end{equation}
  where $Y'=Y \cap (X'+X')$, then
  \begin{equation}\label{eq:size-of-Y-lem}
    (1- \gamma)(r+1) |X'| \leq |Y'| \leq m - \big( 1 - \delta d \sigma \big)\big(d - r\big)k.
  \end{equation}
\end{lem}

We postpone the proof of \Cref{lem:size-of-Y} to the end of this section.
The crucial consequence of this \namecref{lem:size-of-Y} is that it allows us to upper bound the size of $X'$, which in turn yields an upper bound on the size of $X$.
We recall the statement of \Cref{thm:containers} before proving it for convenience.

\containersX*

Next, we prove \Cref{thm:containers} assuming that \Cref{lem:size-of-Y} holds.
The proof begins by applying \Cref{thm:containers-sumset} to the host set $\containersHost \subset \mathbb{Z}^d$.
This produces a family of container pairs $(X, Y)$ of the required size such that every set $B$ of interest is contained in some $X$, while $Y$ is a subset of the sumset $B+B$.
For each pair $(X, Y)$ in this family, the theorem guarantees that either $X$ is very small (which means it cannot even contain our set $B$), or $(X, Y)$ satisfies a ``missing pairs'' property, \eqref{eq:missing-pairs}.

To obtain \eqref{eq:containers-are-small}, the precise bound on $|X|$ that we require, we start from \eqref{eq:missing-pairs}.
We then extract a large subset $X' \subset X$ of $\delta$-robust rank $r \in [d]$ and show that \cref{eq:missing-pairs} implies that the pair $(X', Y')$ (where $Y' = Y \cap (X' + X')$) satisfies \eqref{eq:few-missing-pairs-lem}.
By applying \Cref{lem:size-of-Y} to this pair, we obtain the fundamental inequality
\[(1- \gamma)(r+1) |X'| \leq |Y'| \leq m - ( 1 - \delta d \sigma )(d - r)k.\]
This inequality effectively bounds the size of the robust $X'$, adjusting for the dimensions that were ``lost'' in the robustification process.
Since $X$ is not significantly larger than its robust subset $X'$, choosing $\delta$ appropriately small allows us to conclude that $|X| \le (1 + 5\gamma) \frac{m - (d - r)k}{r+1}$ for some $r \in [d]$, which completes the proof.

\begin{proof}[Proof of \Cref{thm:containers} assuming \Cref{lem:size-of-Y}]
  Fix $0<\gamma < 1/8$.
  In order to apply \Cref{thm:containers-sumset}, note that, by hypothesis, $m \geq (\log t)^2$, where $t=|\Sigma|$.
  Let $c=c(\gamma)$ and $C>0$ be the constants given by \Cref{thm:supsat}, so that we can apply \Cref{thm:containers-sumset} with \(\containersHost\) and
  \[\varepsilon = \sqrt{\frac{c \gamma}{2\sigma d^{C+2}}}<\frac{1}{4},\]
  which follows from $\gamma<1/8$, $c<1$ and $\sigma,d \geq 1$.
  This application results in a container family \(\cF \subset 2^{\containersHost} \times 2^{\containersHost+\containersHost}\) satisfying
  \begin{equation}\label{eq:size-of-container-family}
    |\cF| \leq \exp\bigg(\frac{2^{16}}{\varepsilon^{2}}\sqrt{m( \log t )^{3}}\bigg) \leq \exp\bigg(M d^{C+2}\sqrt{\sigma^3 k (\log t)^3} \bigg),
  \end{equation}
  by our choice of \(\varepsilon\), where $M=M(\gamma) = 2^{17}/(c(\gamma) \gamma)$ depends only on $\gamma$.

  As $\cF$ is the result of an application of \Cref{thm:containers-sumset}, we also have that, for each \(B  \subset \containersHost\) with $|B|=k$ and $|B+B|\leq m$, there is some pair \((X,Y) \in \cF\) so that we have \(B \subset X\), \(Y \subset B + B\) and \((X,Y)\) satisfies either \eqref{eq:small-container} or \eqref{eq:missing-pairs}.
  Without loss of generality, we can assume that for every \((X,Y) \in \cF\), it is a container for some $B$, that is, there exists a $B \in \Lambda_{k,m}^{(d)}$ such that $B \subset X$ and $Y \subset B + B$; otherwise we can remove it from \(\cF\) without affecting the properties of the container family.
  Let $\containers$ be the collection of sets $X$ whose corresponding pair $(X,Y) \in \cF$ satisfies $|X| \geq k$ and that $X$ is full rank.
  It readily follows from \eqref{eq:size-of-container-family} that
  \begin{equation*}
    |\containers| \leq |\cF| \leq \exp\bigg(M d^{C+2}\sqrt{\sigma^3 k (\log t)^3} \bigg).
  \end{equation*}
  We now show that if $X \in \containers$, then it comes from a pair $(X,Y) \in \cF$ that satisfies \eqref{eq:missing-pairs}, that is, if $(X,Y) \in \cF$ satisfies \eqref{eq:small-container}, then $|X|<k$.
  In fact, we have
  \[|X| \leq \frac{m}{\log t} \leq \frac{\sigma k}{\sigma^{2}} \leq \frac{k}{\sigma} < k,\]
  where the second inequality follows from the assumption that $\log t \geq \sigma^2$.
  We can therefore assume that \eqref{eq:missing-pairs} holds for every $(X,Y) \in \cF$ with $|X|\geq k$.
  Thus, it remains to show that \eqref{eq:missing-pairs} implies \eqref{eq:containers-are-small}, namely
  \begin{equation}\label{eq:containers-are-small-rst}
    |X| \leq \max_{r \in [d]} (1 + 5\gamma) \frac{m - (d - r)k}{r + 1},
  \end{equation}
  for every pair \((X, Y) \in \cF\) satisfying \eqref{eq:missing-pairs}.

  To establish that \eqref{eq:containers-are-small-rst} holds, fix $(X, Y) \in \cF$ and let \(X' \subset X\) be a \(\delta\)-robust set given by \Cref{lem:robustify}, where we choose
  \[\delta = \frac{\gamma}{\sigma d^{2}} \leq \frac{1}{2d}.\]
  This implies, in particular, that
  \begin{equation}\label{eq:sizeOfXPrimeFromRobustification}
    |X'| \geq (1 - \delta d)|X| \geq (1-\gamma)|X|.
  \end{equation}
  Now, we would like to apply \Cref{lem:size-of-Y}, and for that we need to verify that $|X| \leq m$ holds and that the pair $(X',Y')$ satisfies \eqref{eq:few-missing-pairs-lem}, where $Y'=Y \cap(X'+X')$.
  We establish that \eqref{eq:few-missing-pairs-lem} holds as a consequence of $(X,Y)$ satisfying \eqref{eq:missing-pairs} and $|X'| \geq (1- \delta d) |X|$.
  In detail:
  \begin{equation*}
    \big| \big\{( x_{1},x_{2}) \in (X')^2~|~x_{1} + x_{2} \notin Y' \big\} \big|
    \leq \frac{c\gamma}{2\sigma d^{C+2}(1 - \delta d)^{2}}|X'|^{2}
    \leq c\delta d^{-C}|X'|^{2}
    \leq c\delta r^{-C}|X'|^{2},
  \end{equation*}
  where in the second to last inequality we used that \((1 - \delta d)^{2} \geq 1/2\), which holds since \(\gamma < 1/8\), and in the last one we relied on $\rank(X') = r \le d$.

  Moreover, we also check that, applying \Cref{thm:supsat} to \(Y\), we obtain
  \begin{equation}\label{eq:whatImplies}
    \big|\big\{( x_{1},x_{2}) \in X^{2}~|~x_{1} + x_{2} \notin Y \big\}\big| \leq \frac{c \gamma}{2\sigma d^{C}}|X|^{2} \leq c\delta d^{-C+2}|X|^{2} \leq c\delta |X|^{2},
  \end{equation}
  observing that \(X\) has \(\delta\)-robust dimension \(1\) and that $|X|\geq k \geq 40 \cdot 2^2/ \gamma$.
  From \eqref{eq:whatImplies}, we conclude that
  \[(1 - \gamma)2|X| \leq |Y| \leq \sigma k\]
  and therefore
  \[|X| \leq \sigma k = m.\]
  Finally, there is a $B \in \Lambda_{k,m}^{(d)}$ with $B \subset X$ and $Y \subset B + B$, since we have discarded container-pairs $(X,Y) \in \cF$ that failed this requirement.

  Having verified that our setting meet the requirements of \Cref{lem:size-of-Y}, we can apply it and conclude that
  \begin{equation}\label{eq:size-of-Y}
    (1 - \gamma)(r + 1)|X'| \leq |Y'| \leq m - \big( 1 - \delta d \sigma \big)(d - r)k.
  \end{equation}
  Rearranging the lower bound, we have
  \begin{equation}\label{eq:lb-on-Y}
    |X'| \leq \frac{|Y'|}{(1 - \gamma)(r+1)}.
  \end{equation}

  In order to prove \eqref{eq:containers-are-small-rst}, let us first show that
  \begin{equation}\label{eq:size-of-X-prime}
    |X'| \leq (1 + 3\gamma)\frac{m - (d - r)k}{r + 1},
  \end{equation}
  Putting \eqref{eq:lb-on-Y} together with the upper bound in \eqref{eq:size-of-Y} and noting that $\sigma d \delta = \gamma/d$ follows from the definition of $\delta$, we obtain
  \begin{equation}\label{eq:size-of-X-prime-1}
    |X'| \leq \frac{1}{(1 - \gamma)(r+1)}\Big(m - \big( 1 - \gamma/d \big)(d - r)k\Big).
  \end{equation}
  We claim that
  \begin{equation}\label{eq:ineq-r}
    m-\big(1-\gamma/d\big)(d-r)k \leq (1+\gamma)(m-(d-r)k),
  \end{equation}
  which, by subtracting $m-(d-r)k$ from both sides and recalling that $m=\sigma k$, is equivalent to
  \[\gamma k(d-r)/d \leq \gamma (\sigma k-(d-r)k)\]
  and
  \begin{equation}\label{eq:inequalityToProve}
  d-r \leq d(\sigma-(d-r)).
  \end{equation}
  In fact, by \Cref{prop:dim-doubling}, $r\geq 1$ and $\sigma \geq d$, we have
  \[\sigma -(d-r) \geq 1,\]
  which establishes \eqref{eq:inequalityToProve}, and therefore we obtain \eqref{eq:ineq-r}.
  Thus, from \eqref{eq:size-of-X-prime-1} and \eqref{eq:ineq-r} we have
  \begin{equation*}
    |X'| \leq \frac{1+\gamma}{1 - \gamma} \, \frac{m-(d-r)k}{r+1}< (1 + 3\gamma)\frac{m - (d - r)k}{r + 1},
  \end{equation*}
  which is exactly \eqref{eq:size-of-X-prime}, where the last inequality follows since $\gamma<1/3$.

  Now, combining \eqref{eq:size-of-X-prime} with \eqref{eq:sizeOfXPrimeFromRobustification}, we obtain
  \begin{equation}\label{eq:equalsTheMaximum}
    |X| \leq \frac{1+3\gamma}{1-\gamma}\, \frac{m - (d - r)k}{r + 1} < (1 + 5\gamma)\frac{m - (d - r)k}{r + 1},
  \end{equation}
  since \(\delta = \gamma/(\sigma d^{2})\) and $\gamma<1/4$.
  As \eqref{eq:equalsTheMaximum} is exactly \eqref{eq:containers-are-small-rst} after taking the maximum over $r \in [d]$, this completes the proof.
\end{proof}

We finish this section with the proof of \Cref{lem:size-of-Y}.
The lower bound on $|Y'|$ in \eqref{eq:size-of-Y-lem} is a direct consequence of \Cref{thm:supsat}, while the upper bound is a geometric argument which roughly says that for each decrement in the dimension incurred by the robustification process, we obtain a factor of $(1-\delta d \sigma)k$ in the final bound.

\begin{proof}[Proof of \Cref{lem:size-of-Y}]
  First, we will establish the lower bound.
  We will apply \Cref{thm:supsat} to the pair $(X',Y')$, and for that, it suffices to check that $X'$ is sufficiently large.
  In fact, since $|X| \geq k \geq 80(d+1)^2/\gamma$, $|X'| \geq (1-\delta d)|X|$ with $1-\delta d \geq 1/2$, and $r \leq d$, we obtain
  \[|X'| \geq \frac{|X|}{2} \geq \frac{40(r + 1)^{2}}{\gamma}.\]
  Thus, by \Cref{thm:supsat}, we have
  \begin{equation*}
    |Y'| \geq (1 - \gamma)(r + 1)|X'|,
  \end{equation*}
  which is exactly the lower bound in \eqref{eq:size-of-Y-lem}.

  We now show that the upper bound in \eqref{eq:size-of-Y-lem} holds.
  First note that it follows from $|X| \leq m$, $m=\sigma k$ and $|X'| \geq (1- \delta d)|X|$ that
  \begin{equation}\label{eq:something}
    |X\setminus X'| \leq \delta d|X| \leq \delta d m = \delta d\sigma k.
  \end{equation}
  Combining \eqref{eq:something} with \(B \subset X\) implies that
  \begin{equation}\label{eq:ub-on-X-prime}
    |B \cap X'| \geq k - |X \setminus X'| \geq (1 - \delta d\sigma)k.
  \end{equation}

  Further observe that, since $B$ is full rank in $\Z^d$, for any set $U \subset \R^d$ such that $\rank(U) < d$, there exists $x \in B$ such that $\rank(U \cup \{x\})=\rank(U) +1$.
  Recalling that $\rank(X') = r\leq d$ and applying the preceding observation recursively, we conclude that there exists a set \[U=\{ x_{1},\ldots,x_{d - r} \} \subset B\setminus X'\] such that
  \[\rank\big(X' \cup \{ x_{1},\ldots,x_{i} \}\big) = r + i\]
  for each $i \in \{ 0,1,\ldots,d - r \}$.
  Since the addition of each $x_i$ spans a new dimension, we have
  \begin{equation}\label{eq:disjoint-translates}
    \big( x_i + B \cap X' \big) \cap (x_j + B \cap X') = \emptyset
  \end{equation}
  for each $i \neq j$, so that
  \begin{equation}\label{eq:size-of-U-plus-B}
    \big|U +  B \cap X'\big| = \sum_{i=1}^{d-r} |x_i+ B \cap X'| \geq \big(1- \delta d \sigma\big)(d-r)k,
  \end{equation}
  where we use \eqref{eq:ub-on-X-prime} in the last inequality.
  Moreover, similarly to \eqref{eq:disjoint-translates}, we have
  \[(x_i + B \cap X') \cap X' = \emptyset,\]
  and, hence\footnote{For sets $A$, $B$ and $C$, we use the notation $A + B \cap C$ to denote $A + (B \cap C)$.},
  \begin{equation*}
    \big( U + B\cap X' \big) \cap (X' + X') = \emptyset.
  \end{equation*}
  Now, recalling definition of $Y'=Y \cap (X'+X') \subset (X'+X')$, we conclude that
  \begin{equation}\label{eq:YPrimeAndUPlusBSectXPrimeAreDisjoint}
    Y' \cap (U + B \cap X') = \emptyset.
  \end{equation}
  Finally, since
  \[ Y'\subset Y\subset B+B, \qquad U+B\cap X'\subset B+B, \]
  and \eqref{eq:YPrimeAndUPlusBSectXPrimeAreDisjoint} holds, we have
  \begin{equation}\label{eq:anotherAboveInequality}
    |Y'|+|U+(B\cap X')|\leq |B+B|\leq m.
  \end{equation}
  Therefore, substituting \eqref{eq:size-of-U-plus-B} into \eqref{eq:anotherAboveInequality}, we obtain
  \[|Y'| \leq m - \big( 1 - \delta d \sigma \big)(d - r)k\]
  which proves the upper bound in \eqref{eq:size-of-Y-lem} and completes the proof.
\end{proof}

\section{Supersaturation}\label{sec:supsat}

In this section, we will turn our attention to the proof of \Cref{thm:supsat}.
First, we will show how to deduce \Cref{thm:supsat} from \Cref{thm:few-translates} and a result of \citet*{jing2023kemperman}, following the analogous steps in \citep{campos2024independence}.
We will then prove the main technical requirement of this section in \Cref{sec:weighted-freiman}, and finally prove \Cref{thm:few-translates} in \Cref{sec:freiman-few-translates}.

Before proving \Cref{thm:supsat}, we state the result of \citet*{jing2023kemperman} that we require.
Although it does not exactly match their statement, it is easy to see that it follows from it by projecting the set $X$ into $\RR^r$.

\begin{thm}[{\cite[Theorem~1.2]{jing2023kemperman}}]\label{thm:jing-mudgal}
  Given \(d,r \in {\mathbb{N}}\), there exists a constant \(M = M(r) > 0\)
  such that, for every set \(X \subset {\mathbb{R}}^{d}\) with
  \(\rank(X) \geq r\), there exist \(x_{1},\ldots,x_{M} \in X\)
  satisfying
  \begin{equation*}
    | {X + \{ x_{1},\ldots,x_{M} \}} | \geq (r + 1)|X| - {5(r + 1)}^{3}.
  \end{equation*}
\end{thm}

\begin{proof}[Proof of \Cref{thm:supsat}]
  Assume that \(|Y| \leq (1 - \gamma)(r + 1)|X|\).
  We iteratively construct a sequence of distinct elements \(x^{(1)},\ldots,x^{(t)} \in X\), where \(t = \big\lceil\delta|X|/4 \big\rceil\), with the following property: for each $i \in \{0,1, \ldots, t-1\}$, the sets \(X_{i} := X\setminus\{ x^{(1)},\ldots,x^{(i)} \}\) satisfy
  \begin{equation}\label{eq:xi-requirements}
    d_i=\rank(X_{i}) \geq r\quad\text{  and  }\quad|( X_{i} + x^{(i + 1)} )\backslash Y| \geq 4c(\gamma)r^{-C}|X|.
  \end{equation}
  The rank condition holds because, for all \(i \leq t\),
  \[|X_{i}| \geq |X| - i \geq |X| - (t-1) \geq \bigg( 1 - \frac{\delta}{4} \bigg)|X|,\]
  so we still have \(\rank(X_{i}) \geq r\), since \(X\) has
  \(\delta\)-robust dimension \(r\).

  To establish the second property of the sets \(X_{i}\) in \eqref{eq:xi-requirements}, we select each \(x^{(i + 1)}\) inductively.
  Suppose we have already selected distinct translates \(\{ x^{(1)},\ldots,x^{(i)} \}\) such that $X_{i}$ satisfies $d_i \geq r$.
  We consider two cases depending on the size of \(\gamma\).

  Suppose first that \(\gamma > 4C\log (2r) / r\).
  In this case, we apply \Cref{thm:few-translates} to \(X_{i}\), which is possible since $\rank(X_{i}) \geq r$ and
  \[|X_{i}| \geq (1 - \delta/4)|X| \geq \frac{20(r + 1)^{2}}{\gamma} \geq \frac{2r}{\gamma}.\]
  Hence, there is a set \(T_{i} \subset X_{i}\) such that
  \[|T_{i}| \leq Cr^{C}\quad\text{  and  }\quad|X_{i} + T_{i}| \geq \bigg( 1 - \frac{\gamma}{4} \bigg)(r + 1)|X_{i}|,\]
  where \(C>0\) is an absolute constant.

  Otherwise, we have \(\gamma \leq 4C\log (2r) / r\) or, equivalently, \(r \leq 8C\log(4 / \gamma) / \gamma\).
  In this case, as
  \[|X_i|\geq (1-\delta/4)\frac{40(r + 1)^{2}}{\gamma} \geq \frac{20(r + 1)^{2}}{ \gamma},\]
  we can apply \Cref{thm:jing-mudgal}, obtaining \(T_{i} = \{ x_{1},\ldots,x_{M} \} \subset X_{i}\) with
  \[|X_{i} + T_{i}| \geq (r+1)|X_i| - 5(r+1)^3 \geq \bigg( 1 - \frac{\gamma}{4} \bigg)(r + 1)|X_{i}|,\]
  for some constant \(M\).
  Importantly, since $r$ is bounded by a function of $\gamma$ in this case, we can take $M=M(\gamma)$ to be a constant depending only on $\gamma$ instead of $r$.
  Consequently, both cases yield a translate set \(T_{i}\) of size \(|T_{i}| \leq MCr^{C}\).

  Combining the bounds, since \(i < t = \big\lceil\delta |X|/4 \big\rceil < \gamma |X|/4 \), we obtain
  \[|X_{i} + T_{i}| \geq \bigg( 1 - \frac{\gamma}{4} \bigg)^{2}(r + 1)|X| \geq \bigg( 1 - \frac{\gamma}{2} \bigg)(r + 1)|X|.\]
  Using our initial assumption that \(|Y| \leq (1 - \gamma)(r + 1)|X|\), it follows that
  \[|( X_{i} + T_{i} )\setminus Y| \geq |X_{i} + T_{i}| - |Y| \geq \frac{\gamma(r + 1)|X|}{2}.\]
  By the pigeonhole principle, some \(x^{(i + 1)} \in T_{i}\) satisfies
  \[|( X_{i} + x^{(i + 1)} )\setminus Y| \geq \frac{\gamma(r + 1)|X|}{2|T_{i}|} \geq 4c(\gamma)r^{-C}|X|,\]
  where $c(\gamma) = \gamma/(8 C M)$.
  Repeating this process for each \(i > 0\) yields the sets \(X_{i}\) and \(\{ x^{(1)},\ldots,x^{(t)} \}\) satisfying \eqref{eq:xi-requirements}.

  Finally, now that we have the sets \(X_{i}\), note each \(X + x^{(i)}\) contributes with at least \(4cr^{-C}|X|\) ordered pairs whose sum are not in \(Y\).
  This gives us a total of
  \[4cr^{- C}|X|t \geq c\delta r^{- C}|X|^{2}\]
  such pairs, because \(t \geq \delta |X|/4\), completing the proof.
\end{proof}

\subsection{Weighted \Freiman's lemma}\label{sec:weighted-freiman}

We now proceed to prove the main technical piece required in the proof of \Cref{thm:few-translates}.
To do so, we adapt ideas from the weighted version of \Freiman's lemma in \citep{campos2024independence} and improve them (see \Cref{prop:phase3}), so we need a few definitions from that paper.

\begin{defi}\label{def:defis-weighted-freiman}
  Let $X\subseteq \RR^d$ be a finite set and $W \subset \RR^d$ a subspace, we define
  \begin{equation*}
    Z(X, W) = \Pi_{W^\perp}(X),
  \end{equation*}
  where $\Pi_U(V)$ denotes the orthogonal projection of $V$ onto $U$.
  We also partition $\R^d$ into equivalence classes in that projection, denoting these by
  \begin{equation*}
    [z]_W = \big\{x \in \R^d : \Pi_{W^\perp}(x) = z\big\},
  \end{equation*}
  and partition X into equivalence classes in the same way:
  \begin{equation*}
    \equivc{z}_{W, X} = [z]_W \cap X.
  \end{equation*}
\end{defi}

It will be convenient to omit the dependence of those definitions on $X$ and $W$ whenever these sets are clear from context.
In this case, the equivalent notation will be simply $Z, [z], \equivc{z}$.
Furthermore, we refer to $[z]$ as a ``fibre'', which we say is ``empty'' if $z \not \in Z$.
From now on, in this section, we will avoid using the notation $[n]$ for the set $\{1, \ldots, n\}$ to avoid confusion.
Below we state our novel version of the weighted version of \Freiman's lemma.

\begin{prop}\label{prop:phase3}
  Let \(d,r \in \N\) and let \(X \subset \R^{d}\) be a finite set.
  Let also \(W \subset \R^{d}\) be a subspace with dimension \(r'\) and \(Z = Z(X,W)\).
  If \(\rank(X) \geq r\), then there is \(T \subset X\)  such that
  \[|X + T| \geq (r - r' + 1)|X| - \binom{|Z| + 1}{2}\]
  and \(|T| \leq 3|Z|^{2}\)
\end{prop}

Compared to \citep[Proposition 4.3]{campos2024independence}, the analogous result in our previous work, we improve the lower bound on $|X + T|$ by a (crucial) factor of 2, in the slightly easier setting where we do not need to avoid $[0]$, the zero fibre.
One of the tools that we use to improve the bound from 1/2 to 1 is the following theorem of \citet*{bollobas2022large}, which can be seen as a sharp one-dimensional version of the result that we require.

\begin{thm}[{\cite[Theorem 1]{bollobas2022large}}] \label{thm:blt}
  Let \(A\) and \(B\) be finite non-empty subsets of \(\ZZ\) with \(|A| \geq |B|\).
  Then there exist elements \(b_{1},b_{2},b_{3} \in B\) such that
  \[\big|A + \{ b_{1},b_{2},b_{3}\}\big| \geq |A| + |B| - 1.\]
\end{thm}

It will be convenient to apply \Cref{thm:blt} in the form below, which follows from the original version by applying a suitable \Freiman isomorphism to both sets.

\begin{cor}\label{cor:blt}
  Let \(A\) and \(B\) be finite non-empty subsets of \(\ZZ^d\) with \(|A| \geq |B|\).
  Then there exist elements \(b_{1},b_{2},b_{3} \in B\) such that
  \[\big|A + \{ b_{1},b_{2},b_{3}\}\big| \geq |A| + |B| - 1.\]
\end{cor}

The final ingredient needed in the proof of \Cref{prop:phase3} is \Cref{prop:weighted-freiman-calc} below.
It provides a lower bound for $|X+T|$, where $T \subset X$, in terms of the sumset of the fibres.
To state it precisely, given a set $Z \subset \R^d$, we define, for each \(z \in Z + Z\), the collection
\[\cS_Z(z) = \Big\{ (x,y) \in Z^{2}\ :\ \big| {\equivc{x}} \big| \geq \big| {\equivc{y}} \big|,\ x + y = z \Big\}.\]

\begin{prop}\label{prop:weighted-freiman-calc}
  Let \(d \in \N\) and let \(X \subset \R^{d}\) be a finite set.
  Let also \(W \subset \R^{d}\) be a subspace and $Z=Z(X,W)$.
  For all $T \subset X$, we have
  \begin{equation*}
    |X + T| \geq \sum_{z \in Z + Z}\max_{(x,y) \in \cS_Z(z)}\big| {\equivc{x} + (\equivc{y} \cap T)} \big|.
  \end{equation*}
\end{prop}

\Cref{prop:weighted-freiman-calc} is proved in \citep{campos2024independence} as a warm-up for a more technical version required in the proof of the weighted version of \Freiman's lemma in that paper.
We reproduce its short proof here for completeness.

\begin{proof}
  Note first that
  \[X + T = \bigcup_{z \in Z + Z} [z] \cap (X + T),\]
  since the projection of \(X\) on \(W^{\perp}\) is exactly \(Z\) and \(T \subset X\).
  In fact, this is a disjoint union, because for any distinct elements \(z,w \in W^{\perp}\) we have \( [z] \cap [w] = \emptyset\).
  Thus,
  \[|X + T| = \sum_{z \in Z + Z}\big|[z] \cap (X + T)\big|,\] so it is enough to prove that
  \begin{equation}\label{eq:reduction-to-fibres}
    \big| {[ z] \cap (X + T)} \big| \geq \max_{(x,y) \in \cS_Z(z)}\big| {\equivc{x} + \equivc{y} \cap T} \big|
  \end{equation}
  holds for every \(z \in Z + Z\).
  For that, we show that for every \(x,y \in Z\) with \(x + y = z\), we have
  \begin{equation}\label{eq:reduction-to-inclusion}
    \equivc{x} + \equivc{y} \cap T \subset [ z] \cap (X + T),
  \end{equation}
  then \eqref{eq:reduction-to-fibres} follows by taking the maximum over all pairs $(x,y) \in Z^2$ with $x+y=z$, \ie, $(x,y) \in \cS_Z(z)$.
  In fact, it follows from $\equivc{y} \cap T \subset [y]$ and $\equivc{x} \subset [x]$ that
  \[\equivc{x} + \equivc{y} \cap T \subset [ x] + [ y] = [ z],\]
  which together with \(\equivc{x} \subset X\) and \(\equivc{y} \cap T \subset T\) yields \eqref{eq:reduction-to-inclusion}.
\end{proof}

Now we turn our attention to the proof of \Cref{prop:phase3}.

\begin{proof}[Proof of \Cref{prop:phase3}]
  Let \(Z = \{ z_{1},\ldots,z_{m} \}\).
  For each pair \(i,j \in \{ 1,\ldots,m \}\) satisfying \(\big| {\equivc{z_{i}}} \big| \geq \big| {\equivc{z_{j}}} \big|\), we apply \Cref{cor:blt} to obtain a set of \(3\) translates \(T_{i,j} \subset \equivc{z_{j}}\) such that
  \[\big| {\equivc{z_{i}} + \equivc{z_{j}} \cap T} \big| \geq \big| {\equivc{z_{i}} + T_{i,j}} \big| \geq | {\equivc{z_{i}}} | + | {\equivc{z_{j}}} | - 1.\]
  Defining
  \begin{equation}\label{eq:defT}
    T = \bigcup_{i = 1}^{m}\bigcup_{j = 1}^{m}T_{i,j},
  \end{equation}
  where $T_{i,j}= \emptyset$ if \(\big| {\equivc{z_{i}}} \big| < \big| {\equivc{z_{j}}} \big|,\) and applying \Cref{prop:weighted-freiman-calc}, we obtain
  \begin{equation*}
    |X + T| \geq \sum_{z \in Z + Z}\max_{(x,y) \in \cS_Z(z)}\big| {\equivc{x} + \equivc{y} \cap T} \big|.
  \end{equation*}
  We now claim that
  \begin{equation}\label{eq:weighted-freiman-bound}
    \sum_{z \in Z + Z}\max_{(x,y) \in \cS_Z(z)} \big| {\equivc{x} + \equivc{y} \cap T} \big| \geq ( \rank(Z) + 1 )|X| - \binom{|Z| + 1}{2},
  \end{equation}
  which, observe, is enough to complete the proof since \(\rank(Z) \geq \rank(X)-\dim(W) \geq r - r'.\)

  Recall that, by the definition of \(T\) in \eqref{eq:defT}, we have
  \begin{equation}\label{eq:blt}
    \big| {\equivc{x} + \equivc{y} \cap T} \big| \geq \big| {\equivc{x}} \big| + \big| {\equivc{y}} \big| - 1
  \end{equation}
  for any \(z \in Z + Z\) and \((x,y) \in \cS_Z(z)\).
  Similarly to the proof of \Freiman's lemma, our proof will proceed by induction on \(|Z|\).
  In the case \(|Z| = 1\), let \(Z = \{ z \}\).
  Observe that $\rank(Z)=0$ in this case, so the right-hand side of \eqref{eq:weighted-freiman-bound} is $|X|-1$.
  On the other hand, the left-hand side is
  \begin{equation*}
    \max_{(x,y) \in \cS_Z(z+z)} \big| {\equivc{x} + \equivc{y} \cap T} \big| =  \big| {\equivc{z} + \equivc{z} \cap T} \big| \geq 2\big| {\equivc{z}} \big| - 1 = 2|X| - 1,
  \end{equation*}
  where the inequality follows from \eqref{eq:blt} and the last equality follows from the definition of $Z$, that is, $X$ is entirely contained in the fibre $[z]$.

  Let \(|Z| \geq 2\) and assume, by induction, that \eqref{eq:weighted-freiman-bound} holds for every \(Z' \subsetneq Z\).
  Fix \(v \in Z\) to be a vertex of the convex hull of \(Z\), and let \(Z' := Z\setminus\{v\}\).
  We first consider the case \(\rank(Z') = \rank(Z) - 1\).
  In this case,
  \begin{equation}\label{eq:we-reference-the-first-term}
    \begin{aligned}
      \sum_{z \in Z + Z} \max_{(x,y) \in \cS_Z(z)} \big| {\equivc{x} + \equivc{y} \cap T} \big|
      & \geq \sum_{z \in Z' + Z'}\max_{(x,y) \in \cS_{Z'}(z)} \big| {\equivc{x} + \equivc{y} \cap T} \big| \\
      & + \sum_{\substack{z \in Z, \\ |\equivc{z}| > |\equivc{v}|}} \big|\equivc{z} + \equivc{v} \cap T \big|
      + \sum_{\substack{z \in Z, \\ |\equivc{z}| \leq |\equivc{v}|}} \big|\equivc{v} + \equivc{z} \cap T \big|
    \end{aligned}
  \end{equation}
  Now, applying the induction hypothesis to the first term of the right-hand side of \eqref{eq:we-reference-the-first-term}, and using \eqref{eq:blt} to bound the other terms, it follows that \eqref{eq:we-reference-the-first-term} is at least
  \begin{equation}\label{eq:conseq-1}
    ( \rank(Z') + 1 )\big( |X| - | {\equivc{v}} | \big) - \binom{|Z|}{2} + \sum_{z \in Z}\big( | {\equivc{z}} | + | {\equivc{v}} | - 1 \big).
  \end{equation}
  Recalling that \(\sum_{z \in Z} | {\equivc{z}} | = |X|\), we can rewrite \eqref{eq:conseq-1} as
  \begin{equation}\label{eq:conseq-2}
    ( \rank(Z') + 2 ) |X|  - \binom{|Z|}{2} - |Z| + |\equivc{v}| \big(|Z|-\rank(Z')-1\big).
  \end{equation}
  Since $\rank(Z')+2=\rank(Z)+1 \leq |Z|$, \eqref{eq:conseq-2} is at least
  \[ ( \rank(Z) + 1 )|X| - \binom{|Z| + 1}{2},\]
  which concludes the case $\rank(Z') < \rank(Z)$.

  Finally, we deal with the case \(\rank(Z') = \rank(Z)\).
  By \Cref{prop:freiman-trick}, there is a set \(U=\{ u_{1},\ldots,u_{\rank(Z) + 1} \} \subset Z\) so that \(v + u_{i} \neq v + u_{j}\) for each \(i \neq j\) and $(v+U) \cap (Z'+Z') = \emptyset$.\footnote{$v$ itself lies in $U$, but this is not important for the argument.}
  Thus,
  \begin{equation}\label{eq:again-we-reference-the-first-term}
  \begin{aligned}
    \sum_{z \in Z + Z}\max_{(x,y) \in \cS_Z(z)} \big| {\equivc{x} + \equivc{y} \cap T} \big|
    & \geq \sum_{z \in Z' + Z'}\max_{(x,y) \in \cS_{Z'}(z)}\big| {\equivc{x} + \equivc{y} \cap T} \big| \\
    & + \sum_{\substack{i \in \{1, \ldots, \rank(Z)+1\}, \\ |\equivc{u_i}| > |\equivc{v}|}} \big| {\equivc{u_{i}} + \equivc{v} \cap T} \big| \\
    & + \sum_{\substack{i \in \{1, \ldots, \rank(Z)+1\}, \\ |\equivc{u_i}| \leq  |\equivc{v}|}} \big| {\equivc{v} + \equivc{u_{i}} \cap T} \big|
  \end{aligned}
  \end{equation}
  Similarly to the previous case, we can apply the induction hypothesis to the first term of \eqref{eq:again-we-reference-the-first-term} to obtain
  \begin{equation}\label{eq:ind-hyp-case-2}
    \sum_{z \in Z' + Z'}\max_{(x,y) \in \cS_{Z'}(z)}\big| {\equivc{x} + \equivc{y} \cap T} \big| \geq ( \rank(Z') + 1 )\big( |X| - | {\equivc{v}} | \big) - \binom{|Z|}{2}.
  \end{equation}
  Moreover, applying \eqref{eq:blt} to the other two terms of \eqref{eq:again-we-reference-the-first-term}, we conclude that their sum is at least
  \begin{equation}\label{eq:blt-case-2}
    \sum_{i = 1}^{\rank(Z) + 1}( | {\equivc{u_{i}}} | + | {\equivc{v}} | - 1 ) \geq ( \rank(Z)+1)| {\equivc{v}} |,
  \end{equation}
  where the last inequality follows from the fact that $|\equivc{u_i}| \geq 1$ for each $i$.
  Therefore, combining \eqref{eq:ind-hyp-case-2} and \eqref{eq:blt-case-2} and recalling that $\rank(Z') = \rank(Z)$ in this case, we obtain
  \begin{equation*}
    \sum_{z \in Z + Z}\max_{(x,y) \in \cS_Z(z)}\big| {\equivc{x} + \equivc{y} \cap T} \big| \geq ( \rank(Z) + 1 )|X| - \binom{|Z|}{2},
  \end{equation*}
  which implies \eqref{eq:weighted-freiman-bound}, and concludes the proof.
\end{proof}

\subsection{\Freiman's lemma via few translates}\label{sec:freiman-few-translates}
Now we turn our attention to the proof of \Cref{thm:few-translates}.
For that, we will first introduce some more useful tools.
One crucial result used in the proof of \Cref{thm:few-translates} is the ``weak'' polynomial \Freiman--Ruzsa conjecture in \(\ZZ^d\), which was recently proved by \citet*{gowers2023conjecture}.
In our approach, this replaces the simple greedy geometric argument used in \citep{campos2024independence}; this replacement is essential because the greedy argument loses a factor of $2$, which we cannot afford in the present work.

\vspace{-0.1pt}

\begin{thm}[{\citep[Theorem 1.3]{gowers2023conjecture}}]\label{thm:weak-pfr}
  There exist \(C_{\dim},C_{\pfr} > 0\) such that the following holds for all \(d \in {\mathbb{N}}\) and finite sets \(X \subset {\mathbb{Z}}^{d}\).
  If \(|X + X| \leq \sigma|X|\), then there is \(X_{0} \subset X\) so that
  \[|X_{0}| \geq \sigma^{- C_{\pfr}}|X|\]
  and
  \[\rank(X_{0}) \leq C_{\dim}\log\sigma.\]
\end{thm}

We will also need the following result of \citet*{fox2022small}, which will be used in the case that the rank of the set is small compared with its doubling.

\begin{thm}[{\citep[Theorem 1.1]{fox2022small}}] \label{thm:flpz}
  There exists \(c > 0\) such that the following holds.
  Let \(G\) be an abelian group.
  Let \(X \subset G\) be nonempty and let \(\sigma = \sigma[X]\).
  Then for each \(s \geq 1\), there exists \(T \subset X\) of size at most \(s\) such that
  \[| {X + T} | \geq c\min\big\{\sigma^{1/3},s\big\}|X|.\]
\end{thm}

The last result that we require is from our previous work~\citep{campos2024independence}.

\begin{prop}[{\citep[Proposition 4.2]{campos2024independence}}]\label{prop:Z-large}
  Let \(d,r \in {\mathbb{N}}\), \(\beta > 0\), and let \(X \subset {\mathbb{R}}^{d}\) be a finite set with \(\rank(X) \geq r\). If there exists a subspace \(W \subset {\mathbb{R}}^{d}\) such that \(|X \cap W| \geq \beta|X|\) and \(Z = Z(X,W)\) satisfies \(|Z| \geq (r + 1)/\beta\) then there is \(T \subset X\) such that
  \[|X + T| \geq (r + 1)|X|\]
  and \(|T|= (r + 1)/\beta\).
\end{prop}

We can now prove \Cref{thm:few-translates}.

\begin{proof}[Proof of \Cref{thm:few-translates}]
  Set \(\sigma=\sigma[X]\).
  First, we handle the case where \(r + 1 \leq c \sigma^{1/3}\) with $c$ as in \Cref{thm:flpz}.
  Applying \Cref{thm:flpz} to \(X\) with \(s = (r + 1)/c\), we obtain a set \(T \subset X\) such that
  \[|T| \leq s=\frac{r + 1}{c} \qquad \text{and} \qquad | {X + T} | \geq \min\{r+1,c\sigma^{1/3}\}|X| = (r + 1)|X|,\]
  and hence the proof is complete in this case.

  We can now assume that \(r \ge c\sigma^{1/3}-1\).
  As $r \geq 1$, it follows that
  \begin{equation}\label{eq:lb-in-r}
    r \geq \frac{c \sigma^{1/3}}{2}.
  \end{equation}
  Applying \Cref{thm:weak-pfr} to $X$, we obtain a subset \(X_{0} \subset X\) such that
  \[|X_{0}| \geq \sigma^{- C_{\pfr}}|X|\]
  and
  \[\rank X_0 \leq C_{\dim}\log\sigma \leq 3C_{\dim}\log (C'r),\]
  where the last inequality follows from \eqref{eq:lb-in-r} with $C' = 2/c$.
  Let \(W \subset \R^{d}\) be a subspace\footnote{To see that such subspace $W$ exists, take the minimal affine subspace given by the definition of $\rank(X_0)$ and (if necessary) extend it to a subspace by adding the element zero to it and taking the span.} such that $ X_0 \subset W$ and
  \begin{equation*}
    \dim W \leq 3C_{\dim}\log (C'r)+1 \leq M \log (2r)
  \end{equation*}
  for $M>0$ a sufficiently large constant.

  Now let $Z=Z(X,W)$ (see \Cref{def:defis-weighted-freiman}), and consider two cases.
  If \(|Z| \geq (r + 1)\sigma^{C_{\pfr}}\), we can apply \Cref{prop:Z-large} with \(\beta = \sigma^{- C_{\pfr}}\) to obtain \(T \subset X\) with
  \[|T| = (r + 1)\sigma^{C_{\pfr}} = O( r^{3C_{\pfr} + 1} ),\] such that \[| {X + T} | \geq (r + 1)|X|,\] as desired.
  In the other case, note first that we have
  \begin{equation}\label{eq:ub-in-Z}
    |Z| \leq (r + 1)\sigma^{C_{\pfr}} \leq M'r^{3C_{\pfr} + 1},
  \end{equation}
  by \eqref{eq:lb-in-r} if we choose $M'=2(2/c)^{3C_{\pfr}}$.
  Therefore, applying \Cref{prop:phase3}, we obtain \(T \subset X\) such that
  \[|T| \leq 3|Z|^{2} = O( r^{6C_{\pfr} + 2} )\]
  and
  \begin{equation*}
    |X + T| \geq (r +1 - M\log (2r))|X| - \binom{|Z| + 1}{2} \geq (1 - \gamma/2)(r + 1)|X|- \binom{|Z|+1}{2},
  \end{equation*}
  where the last inequality holds since \(\gamma \geq 2M \log (2r)/r\).

  We are now done if
  \[\binom{|Z| + 1}{2} \leq \frac{\gamma(r + 1)|X|}{2}.\]
  This follows from \eqref{eq:ub-in-Z} if $|X| \geq M''r^{6C_{\pfr}+2}$ holds, for $M''$ chosen sufficiently large.

  On the other hand, if $|X| < M''r^{6C_{\pfr}+2},$ then we can simply set $T = X$ and, by \Freiman's lemma (\Cref{stmt:freimansLemma}), we have
  \[|X + T| \geq (r + 1)|X| - \binom{r+1}{2} \geq (r + 1)|X| - (r+1) \gamma|X| = (1 - \gamma)(r + 1)|X|,\]
  where the last inequality follows from our assumption that $|X| \geq r /(2\gamma)$.
\end{proof}

\section{A full rank version of \Freiman's theorem}\label{sec:fullrank-freimans-thm}

In this section we will prove \Cref{thm:fullrank-freimans-thm}, and hence complete the proof of the upper bound in \Cref{thm:main}.
We restate \Cref{thm:fullrank-freimans-thm} for convenience, and refer the reader to \Cref{sec:main} for the definitions of \Freiman isomorphism and \Freiman dimension.

We will deduce \Cref{thm:fullrank-freimans-thm} from a version of \Freiman's theorem, proved by \citet{chang2002polynomial}, that yields proper progressions.
Recall first that a generalised arithmetic progression (GAP for short) $P$ is a set of the form
\begin{equation*}
  P=\bigg\{ a_0 + \sum_{i = 1}^d w_i a_i : 0 \le w_i \le \ell_i-1 \bigg\}
\end{equation*}
for some $a_0, \ldots, a_d \in \ZZ$ and $\ell_1, \ldots, \ell_d \in \N$, where $d=\dim(P)$ is the \emph{dimension} of $P$.
Moreover, we say that $P$ is proper if every element of $P$ has a unique representation of the above form; equivalently, $P$ is proper if $|P|=\prod_{i=1}^d \ell_i$.
Finally, we say $P$ is $2$-proper if $P+P$ is proper.

\begin{thm}[{\cite[Proposition 1.3 with $t=2$]{green2006compressions}}]\label{thm:chang}
  There exists an absolute constant $C > 0$ such that the following holds.
  Let $A$ be a finite subset of $\ZZ$ with $|A| > 2$ and $\doubling[A] = \doubling$.
  Then $A \subset P$ where $P$ is a $2$-proper GAP satisfying
  \begin{equation}\label{eq:changs-thm}
    \dim(P) \le 2\doubling \qquad \text{ and } \qquad |P| \le \exp(C \sigma^2 (\log \sigma)^3) |A|.
  \end{equation}
\end{thm}

To deduce \Cref{thm:fullrank-freimans-thm}, we will embed the $2$-proper GAP given by \Cref{thm:chang} into $\ZZ^d$ for some $d \leq 2\sigma$ via a natural map with an affine inverse (see \eqref{eq:def-of-phi}).
The problem now is that the image of $A$ via this map might not be full rank.
For instance, if the GAP has dimension $2$ but the \Freiman dimension of $A$ is larger, then the image of $A$ under this map might still lie in a line, and thus have rank $1$, failing \cref{item:fullrank-freimans-thm} in \Cref{thm:fullrank-freimans-thm}.
To overcome this obstacle, we will use the following \namecref{stmt:hunter-stronger-proper-gap} of \citet{green2022new}.

\begin{lem}[{\cite[Lemma A.2]{green2022new}}]\label{stmt:hunter-stronger-proper-gap}
  Let $q \ge 1$ be a parameter.
  Let $V \le \QQ^d$ be a subspace spanned over $\QQ$ by linearly independent vectors $v^{(1)}, \ldots, v^{(m)} \in \ZZ^d$ where $\norm{v^{(i)}}_\infty \le q$ for each $i \in [m]$.
  Then there is an integral basis $w^{(1)}, \ldots, w^{(m)} \in V \cap \ZZ^d$ such that every $x \in V \cap \ZZ^d$ with $\norm{x}_\infty \le q$ is a (unique) $\ZZ$-linear combination $x = \sum_{i = 1}^m n_i w^{(i)}$ with $|n_i| \le m! (2 q)^m$.
\end{lem}

After embedding the GAP $P$ into $\ZZ^{d}$ in the proof, we apply \Cref{stmt:hunter-stronger-proper-gap} to the minimal affine subspace that contains the image of $A$ via this map.
In this way, we obtain a new basis for this affine subspace such that every element of $A$ has a unique representation as an integer linear combination of these new basis vectors with bounded coefficients, which gives us the desired bound in \eqref{eq:fullrank-freimans-thm-size-of-box}.

To obtain better bounds, we could replace \Cref{thm:chang} with the result of \citet{sanders2012bogolyubov}, or the very recent result of \citet{raghavan2025improved}, together with the procedure of Green's notes \citep{green2005notes} to obtain proper progressions from general ones (see \cite[Section~4]{nenadov2025remark}, where this is done for cyclic groups).
However, the bounds obtained in this way would not significantly improve the quantitative aspect of our main result, so we have chosen to use the more classical statement of \Cref{thm:chang}.

\begin{proof}[Proof of \Cref{thm:fullrank-freimans-thm}]
  We apply \Cref{thm:chang} to $A$ and obtain $P$, a $2$-proper GAP satisfying \eqref{eq:changs-thm}.
  Let $d_0 = \dim(P) \le 2\sigma$, $s_0=|P| \le \exp(C \sigma^2 (\log \sigma)^3) k$ and define the function
  \[\phi_P : P \to P_0  \subset \ZZ^{d_0},\]
  with $P_0 = \phi_P(P)$, as the natural embedding of $P$ into $\ZZ^{d_0}$, \ie,
  \begin{equation}\label{eq:def-of-phi}
    \phi_P\bigg(a_0 + \sum_{i = 1}^{d_0} w_i a_i\bigg) = (w_1, \ldots, w_{d_0}).
  \end{equation}
  Observe this is well-defined since $P$ is $2$-proper and, therefore, each element of $P$ has a unique representation as in \eqref{eq:def-of-phi}.
  Moreover, $P_0 \subset [0, s_0]^{d_0}$ because whenever $a_0+ \sum_{i=1}^{d_0} w_i a_i \in P$ we have $0 \le w_i \le \ell_i - 1$ and $\ell_i \le \prod_{j=1}^{d_0}\ell_j = |P| = s_0$ for each $i \in [d_0]$.
  Note also that $\phi_P$ is a \Freiman isomorphism onto its image, since it is a bijection whose inverse is affine and $P$ is $2$-proper.
  Let also
  $A_0 = \phi_P(A) \subset \ZZ^{d_0}$, and observe that
  \begin{equation}\label{eq:infty-norm-a}
    \norm{a}_\infty \le s_0 \quad \text{ for every } a \in A_0.
  \end{equation}

  Take $V \le \QQ^{d_0}$ to be a minimal dimension subspace such that there exists $v^{(0)} \in A_0$ with $A_0 \subset v^{(0)}+V$ which, in particular, satisfies $\rank(A_0) = \dim(V) = d$.
  As $A_1 = A_0 - v^{(0)}$ is full rank in $V$, we can find linearly independent $v^{(1)}, \ldots, v^{(d)} \in A_1 \subset \ZZ^{d_0}$ that span $V$ and, as a consequence of \eqref{eq:infty-norm-a}, satisfy $\norm{v^{(i)}}_\infty \le 2 s_0$ for every $i \in [d]$.

  We can then apply \Cref{stmt:hunter-stronger-proper-gap} with $m = d$ and $q = 2 s_0$ to obtain $w^{(1)}, \ldots, w^{(d)} \in V \cap \ZZ^{d_0}$ such that every $x \in V \cap \ZZ^{d_0}$ with $\norm{x}_\infty \le 2 s_0$ is a unique $\ZZ$-linear combination
  \begin{equation*}
    x = \sum_{i = 1}^{d} n_i w^{(i)}
  \end{equation*}
  with $|n_i| \le (4 d s_0)^{d}$ for all $i \in [d]$.
  In particular, this holds for every $a' \in A_1$, so that we can uniquely write
  \begin{equation}\label{eq:unique-rep-for-a'-in-A_1}
    a' = a - v^{(0)} = \sum_{i = 1}^{d} n_i w^{(i)}
  \end{equation}
  with $n_i \in \ZZ$ and $|n_i| \le (4 d s_0)^{d}$ for all $i \in [d]$.
  The uniqueness of the solutions in \eqref{eq:unique-rep-for-a'-in-A_1} for every $a' \in A_1$ allows us to define the \Freiman isomorphism $\psi : A_0 \to A_2 \subset \ZZ^{d}$ as
  \begin{equation*}
    \psi(a) = (n_1, \ldots, n_{d})
  \end{equation*}
  for each $A_0 \ni a = v^{(0)} + \sum_{i = 1}^{d} n_i w^{(i)}$ with $n_i$ as in \eqref{eq:unique-rep-for-a'-in-A_1}.

  Define now
  \[\varphi:=(\phi_P|_A)^{-1} : A_0 \to A.\]
  We claim that we can take
  \[\phi = \varphi \circ \psi^{-1} : A_2 \to A,\]
  $s = (4 d s_0)^{d}$ and $B = A_2$ to satisfy the statement of \Cref{thm:fullrank-freimans-thm}.
  Since $\psi^{-1}$ and $\varphi$ are affine bijections, $\phi$ is an affine bijection, as desired, and, therefore, it is also a \Freiman isomorphism.

  Observe that $B \subset [-s, s]^d$ by the definition of $\psi$, and furthermore that we have $$d \le \dimF(A) \le 2 \doubling.$$
  Here, the first inequality follows since $B \subset \ZZ^d$ is full rank and \Freiman isomorphic to $A$, so $d \leq \dimF(A)$ by definition of \Freiman dimension.
  The second inequality is due to \Cref{prop:dim-doubling}.
  Thus, it remains only to bound $s$ by
  \begin{equation*}
    s = (4 d s_0)^{d} \le \big(8 \doubling  \exp(C \doubling^2 (\log \doubling)^3) \, k\big)^{2\doubling} \le \exp\big(C' (\doubling \log \doubling)^3\big)k^{2\doubling}
  \end{equation*}
  where the first inequality follows from \eqref{eq:changs-thm} and the definition of $s_0$, and the second from taking an absolute constant $C' > 0$ suitably larger than $C$, completing the proof.
\end{proof}

\section{Lower bound construction}\label{sec:construction}

In this section we will prove the lower bound of \Cref{thm:main}.
In fact, we will prove it for a wider range of parameters, as stated in the following theorem, and then show that it implies the bound in \Cref{thm:main}.
Recall that we defined
\[\lambda(d)=\lambda_{k,m}(d) = \bigg\lfloor\frac{m-(d-1)k + \binom{d+1}{2}}{2}\bigg\rfloor-d+1.\]
\begin{thm}\label{thm:lower-bound}
  Let $n,k, m \in \N$.
  If $k \leq m \leq n/4$, then
  \begin{equation}\label{eq:lower-bound}
    \big|\Lambda_{n,k,m}\big| \geq \sum_{d} \frac{n^{d+1}}{64m(2d)^d}  \binom{\lambda(d)-2}{k-d-1},
  \end{equation}
  where the sum in \eqref{eq:lower-bound} is over $d \in [k-1]$ satisfying
  \begin{equation}\label{eq:condition-on-d-lb}
    (d+1)k -\binom{d+1}{2} \leq m.
  \end{equation}
\end{thm}
The idea of the proof is simple: we will choose a set of size $k-(d-1)$ inside an arithmetic progression $P\subset [n/2]$ of size $\lambda(d)$ and let the remaining $d-1$ elements be arbitrary points of the set $\{n/2+1, \ldots, n\}$.
Typically, we expect each of these $d-1$ elements to contribute with $k$ to the size of the sumset.
Below we formalise this idea.
\begin{proof}
  Fix $d \in [k-1]$ satisfying \eqref{eq:condition-on-d-lb} and observe that this implies that
  \begin{equation}\label{eq:lb-lambda}
    \lambda(d) \geq k-d+1.
  \end{equation}
  Choose $1 \leq a\leq n/4$ and $1 \leq b \leq n/(4m)$ arbitrarily and define the arithmetic progression
  \[P=\Big\{a+ib : 0 \leq i \leq \lambda(d) -1\Big\}.\]
  Note that there are $\lfloor n/(4m) \rfloor (n/4)  \geq n^2/(64m)$ choices for the progression $P$, and that we have
  \[\max P \leq a+ m b \leq n/2,\]
  by our choice of $a$ and $b$ and since $\lambda(d) \leq m$, and hence $P \subset [n/2]$.

  Now, choose $T \subset [n]\setminus [n/2]$ arbitrarily with $|T|=d-1$ and note that the number of possible choices for $T$ is at least $\binom{n/2}{d-1} \geq (2d)^{-d}n^{d-1}$.
  Finally, we choose $A_0 \subset P$ with $|A_0|=k-d+1$ such that $a,a+b \in A_0$, where $A_0$ fits in $P$ by \eqref{eq:lb-lambda}.
  Given $P$, then, the number of possible choices for $A_0$ is at least
  \[\binom{\lambda(d)-2}{k-d-1}.\]

  Define $A=A_0 \cup T$, so that $|A|=k$.
  Now, let us show that $|A+A| \leq m$.
  In fact, the size of the sumset is at most
  \begin{equation}\label{eq:bound-sumset}
    |A+A|\leq |A_0+A_0|+|T+A_0|+|T+T|,
  \end{equation}
  so we bound each term from above separately.
  From $A_0+ A_0 \subset P+P$ and recalling that $P$ is an arithmetic progression, we have
  \[|A_0+A_0| \leq |P+P| = 2|P|-1 \leq m-(d-1)k+\binom{d+1}{2}-2d+1.\]
  We also trivially have that
  \[|T+T|\leq \binom{|T|}{2}+|T|=\binom{d}{2}\]
  and
  \[|T+A_0| \leq |T| \, |A_0| \le (d-1)(k-d+1) = (d-1)k - (d-1)^2,\]
  and substituting these in \eqref{eq:bound-sumset}, we obtain $|A+A| \leq m$, since
  \[ \binom{d+1}{2} - 2d + 1 + \binom{d}{2} - (d-1)^2 = 0.\]
  Thus, $A \in \Lambda_{n,k,m}$.

  The preceding process defines, for each $d \in [k-1]$ satisfying \eqref{eq:condition-on-d-lb}, a triple $(P,T,A)$ with $A \in \Lambda_{n,k,m}$.
  Furthermore, a lower bound on the number of choices for this triple is exactly the summand in \eqref{eq:lower-bound}.
  Therefore, to finish the proof, it is enough to show that we can determine $d$, $T$ and $P$ from $A$.
  To obtain $T$ and $d$, we simply take $T=A \setminus [n/2]$ and $d=|T|+1$.
  Since $d$ is already determined, recovering $P$ from $A$ reduces to determining $a$ and $b$.
  Recall that, as a consequence of our construction, $A$ contains both $a$ and $a+b$, and that these are its two smallest elements.
  Hence, $a$ is the smallest element of $A$, while $b$ is the difference between the second smallest element of $A$ and $a$.
\end{proof}

Now let us check that \Cref{thm:lower-bound} implies the lower bound in \Cref{thm:main}.
\begin{proof}[Proof of the lower bound in \Cref{thm:main}]

  First, observe that we have exactly the same condition \eqref{eq:condition-on-d-lb} on $d \in [k-1]$ in \Cref{thm:lower-bound} and in \Cref{thm:main}, so that the sum is over the same values of $d$ in both theorems.
  Now, to obtain the right binomial term in the bound, note that
  \[\binom{\lambda(d)-2}{k-d-1} = \frac{k-d}{\lambda(d)-1} \, \frac{k-d+1}{\lambda(d)}  \, \binom{\lambda(d)}{k-d+1} \geq \frac{1}{m^2} \binom{\lambda(d)}{k-d+1},\]
  where the last inequality follows from $\lambda(d) \leq m$ and $d \leq k-1$.
  Therefore, from \eqref{eq:lower-bound}, we have
  \begin{equation}\label{eq:intermediate-lb}
    |\Lambda_{n,k,m}| \geq \sum_{d} \frac{1}{64m^3(2d)^d} n^{d+1} \binom{\lambda(d)}{k-d+1},
  \end{equation}

  We also have $m \leq k^{1+\alpha} \leq 2^{o(k)}$, which implies $64m^3 \leq 2^{o(k)}$.
  Furthermore, \Cref{prop:dim-doubling} implies that $d \leq \sigma \leq k^\alpha$, and hence $(2d)^d \leq 2^{o(k)}$.
  Substituting these estimates into \eqref{eq:intermediate-lb} yields
  \[|\Lambda_{n,k,m}| \geq 2^{o(k)}\sum_{d} n^{d+1} \binom{\lambda(d)}{k-d+1},\]
  as required.
\end{proof}

\section*{Acknowledgements}

We are grateful to Rob Morris for his careful reading of this paper and for his valuable suggestions, which have significantly improved its presentation.
We would also like to thank Zach Hunter for bringing \Cref{stmt:hunter-stronger-proper-gap} to our attention.

During this work, Marcelo Campos was supported by Serrapilheira (grant R-2412-51283), and Jo\~{a}o Pedro Marciano was supported by a FAPERJ Bolsa Nota 10.

\section*{AI usage statement}
The work leading to the proof of \Cref{thm:main} was carried out by the authors in 2025, without the use of any AI tools.
The only use of AI during the writing of the paper was to proofread the final version.

\renewcommand*{\bibfont}{\normalfont\small}
\printbibliography

\appendix
\section{Binomial calculations}\label{sec:appendix-calc}

In this appendix, we establish the following inequality, which is used in the proof of \Cref{lem:reduction-to-Zd}.
\begin{lem}\label{lem:binom-main}
  Let $m,k,d \in \NN$ be as in \Cref{thm:main}, let $k^{-1/2}<\gamma<1/5$.
  Then we have
  \begin{equation}\label{eq:binom-main}
    \max_{r \in [d]} \binom{(1 + 5\gamma)\frac{m - (d - r)k}{r + 1}}{k} \leq 2^{30\gamma \log(1/\gamma) k}\binom{\lambda_{k,m}(d)}{k-d+1}.
  \end{equation}
\end{lem}
Recall
\begin{equation*}
  \lambda(d)=\lambda_{k,m}(d) = \bigg\lfloor\frac{m-(d-1)k + \binom{d+1}{2}}{2}\bigg\rfloor-d+1.
\end{equation*}
This lemma follows from simple calculations involving binomial terms, as we now show.
We separate the prove \Cref{lem:binom-main} in two steps, namely \Cref{thm:binom-opt} and \Cref{thm:binom-simp}.
First, we show that the left-hand side of \eqref{eq:binom-main} with $r=1$ is close to the maximum.

\begin{lem}\label{thm:binom-opt}
  Let $m,k,d \in \NN$ be as in \Cref{thm:main} and let $k^{-1/2}<\gamma<1/5$.
  Then we have
  \begin{equation}\label{eq:binom-opt}
    \max_{r \in [d]}\binom{(1+5 \gamma)\frac{m - (d - r)k}{r + 1}}{k} \leq 2^{10\sigma^2\log(1/\gamma)} \binom{(1+5 \gamma)\frac{m - (d - 1)k}{2}}{k}.
  \end{equation}
\end{lem}

\begin{proof}
  Note that
  \[M(r)=(1+5 \gamma)\frac{\sigma - d +r}{r+1}\]
  does not grow fast with $r \in [d]$.
  Defining $\eta=d+1-\sigma$, so that 
  \[M(r) = (1+5\gamma)\frac{r+1-\eta}{r+1}=(1+5 \gamma)\bigg(1-\frac{\eta}{r+1}\bigg),\] 
  we consider two cases.
  If $\eta \leq 0$, then, taking the ratio of consecutive terms, for $r \geq 2$ we have
  \begin{equation*}
    \frac{M(r)}{M(r-1)} =  \frac{r+1-\eta}{r+1}\cdot \frac{r}{r-\eta}<1.
  \end{equation*}
  Therefore, $M(r)$ is decreasing in $r$ and the maximum is attained at $r=1$.
  
  Now, consider the case $\eta > 0$.
  By \Cref{prop:dim-doubling}, we have $\eta \leq \frac{2\sigma^2}{k}$, and therefore
  \[M(r)-M(1) = (1+5 \gamma)\eta\left( \frac{1}{2}-\frac{1}{r+1} \right) \leq \frac{2\sigma^2}{k}.\]
  Moreover, since $\gamma >k^{-1/2}$ and $\sigma \leq k^{\alpha}$ for $\alpha<1/4$ implies
  \[\frac{\sigma^2}{k} \leq k^{2 \alpha-1} \leq k^{-1/2}< \gamma,\]
  we obtain
  \[M(1)=(1+5 \gamma)(1-\eta/2) \geq (1+5 \gamma)\bigg(1-\frac{\sigma^2}{k}\bigg) \geq (1 +5\gamma)(1- \gamma) \geq 1+3 \gamma.\]
  Thus, combining the above inequalities, we have
  \begin{equation*}
    \max_{r \in [d]} \binom{M(r)k}{k} \leq \binom{M(1)k + 2\sigma^2}{k} \leq 2^{10\sigma^2\log(1/\gamma)}\binom{M(1)k}{k},
  \end{equation*} 
  as we wanted to show.
\end{proof}

Now we proceed to show that the binomial term obtained in the right-hand side of \eqref{eq:binom-opt} is at most $2^{20\gamma \log(1/\gamma) k}$ times $\binom{\lambda(d)}{k - d +1}$.

\begin{lem}\label{thm:binom-simp}
  Let $m,k,d \in \NN$ be as in \Cref{thm:main}, let $r \in [d]$ and let $k^{-1/2}<\gamma<1/5$.
  Then we have
  \begin{equation*}
    \binom{(1 + 5\gamma)\frac{m - (d - 1)k}{2}}{k} \leq 2^{20\gamma \log(1/\gamma) k}\binom{\lambda(d)}{k - d +1}.
  \end{equation*}
\end{lem}

\begin{proof}
  Note first that, defining $M=\frac{\sigma-(d-1)}{2}$, we have
  \[\frac{m-(d-1)k}{2} = \frac{\sigma-(d-1)}{2}k = M k.\]
  In the case $M < 1$, we have
  \[\binom{(1+5 \gamma ) Mk}{k} \leq \binom{(1+5 \gamma ) k}{k} = \binom{(1+5 \gamma)k}{5\gamma k} \leq 2^{10\, \gamma \log (1/\gamma) k},\]
  which, since
  \[\binom{\lambda(d)}{k - d +1} \geq 1,\]
  gives us
  \[\binom{(1 + 5\gamma)\frac{m - (d - 1)k}{2}}{k} \leq 2^{20\, \gamma \log (1/\gamma) k}\binom{\lambda(d)}{k - d +1},\]
  and we are done in this case.

  Consider now the case $1 \leq M \leq 2$.
  In this case, we bound the ratio instead:
  \[ \frac{\binom{(1+5\gamma)Mk}{k}}{\binom{Mk}{k}} = \prod_{i=1}^{5\gamma Mk} \frac{Mk+i}{Mk-k+i} \leq \prod_{i=1}^{5\gamma Mk} \frac{2k+i}{i} = \binom{2k+5\gamma Mk}{5\gamma Mk}. \]
  Since $$\lambda(d)+d-1 = Mk + \frac{1}{2}\binom{d+1}{2} \geq Mk,$$ we know that $$\binom{Mk}{k} \leq \binom{\lambda(d)+d-1}{k},$$
  which is always at least 1.
  It then follows that, for all $1 \leq M \leq 2$,
  \[ \binom{(1+5\gamma)Mk}{k} \leq \binom{2k+5\gamma Mk}{5\gamma Mk} \binom{\lambda(d)+d-1}{k}. \]
  Using the bound $\binom{n}{k} \leq \left(\frac{en}{k}\right)^k$ and the fact that $\gamma < 1/5$, we obtain:
  \[ \binom{2k+10\gamma k}{10\gamma k} \leq \left(\frac{e(2+10\gamma)}{10\gamma}\right)^{10\gamma k} \leq \left(\frac{2}{\gamma}\right)^{10\gamma k} \leq 2^{15\gamma \log(1/\gamma) k}.\]
  This gives us our intermediate bound for $1\leq M \leq 2$:
  \begin{equation}\label{eq:intermediate-bound}
    \binom{(1+5\gamma)Mk}{k} \leq 2^{15\gamma \log(1/\gamma) k}\binom{\lambda(d)+d-1}{k}.
  \end{equation}

  Finally, we bound $\binom{\lambda(d)+d-1}{k}$ from above in terms of $\binom{\lambda(d)}{k-d+1}$ as follows:
  \begin{equation}\label{eq:removing-dminusone}
    \binom{\lambda(d)+d-1}{k} \leq (2M)^{d-1} \binom{\lambda(d)}{k - d + 1}  \leq \sigma^{\sigma} \binom{\lambda(d)}{k - d + 1} \leq 2^{5\gamma \log(1/\gamma) k} \binom{\lambda(d)}{k - d + 1},
  \end{equation}
  where in the first inequality we use
  \[\frac{\lambda(d)}{k-d+1} \leq 2M,\]
  and in the last inequality we use $\sigma \leq k^{\alpha}$ and $\gamma \geq k^{-1/2}$.
  Substituting \eqref{eq:removing-dminusone} into \eqref{eq:intermediate-bound} yields
  \[\binom{(1 + 5\gamma)\frac{m - (d - 1)k}{2}}{k} \leq 2^{20\gamma \log(1/\gamma) k}\binom{\lambda(d)}{k - d +1},\]
  establishing the desired bound for $1 \leq M \leq 2$. 

  Therefore, it is enough to consider the case $M > 2$, where we need a bit more careful calculation.
  Note first that for any $a,b \in \N$ with $0 \leq b \leq a$ we have
  \[ \binom{a+1}{b} \leq \bigg(1+\frac{b}{a-b+1}\bigg) \binom{a}{b}.\]
  Applying this recursively for
  \[t=5\gamma M k -\frac{1}{2}\binom{d+1}{2} \leq 5 \gamma M k\]
  steps and recalling that
  \[\lambda(d)+d-1=Mk+\frac{1}{2}\binom{d+1}{2}\]
  we get
  \begin{align*}
    \binom{(1+5\gamma)Mk}{k} &\leq  \prod_{i=1}^{5\gamma Mk} \bigg( 1+ \frac{k}{(M-1)k+i}\bigg)\binom{\lambda(d)+d-1}{k}.
  \end{align*}
  Substituting \eqref{eq:removing-dminusone} into the above inequality, it is enough to show that
  \begin{equation*}
    \prod_{i=1}^{5\gamma Mk} \bigg( 1+ \frac{k}{(M-1)k+i}\bigg) \leq 2^{15\gamma k} \leq 2^{15\gamma \log(1/\gamma) k}.
  \end{equation*}
  For that, we use $1+x \leq e^x$ for $x \geq 0$ in the above product to obtain
  \begin{equation*}
    \exp\bigg(k \sum_{i=1}^{5 M \gamma k }\frac{1}{(M-1)k+i}\bigg)  \leq \exp\Bigg(k \log\bigg(1+ 5\gamma + \frac{5\gamma}{M-1}\bigg)\Bigg) 
  \end{equation*}
  Finally, using $M>2$ in the above right-hand side, we obtain
  \[\prod_{i=1}^{5\gamma Mk} \bigg( 1+ \frac{k}{(M-1)k+i}\bigg) \leq \exp\Bigg(k \log\bigg(1+ 10\gamma\bigg)\Bigg)\leq 2^{15\gamma k},\]
  as needed to complete the proof.
\end{proof}

Now, we combine \Cref{thm:binom-opt} and \Cref{thm:binom-simp} to prove \Cref{lem:binom-main}.
\begin{proof}[Proof of \Cref{lem:binom-main}]
  Applying \Cref{thm:binom-opt} and then \Cref{thm:binom-simp}, we have
  \begin{equation*}
    \max_{r \in [d]} \binom{(1 + 5\gamma)\frac{m - (d - r)k}{r + 1}}{k} \leq 2^{10\sigma^2\log(1/\gamma)} \binom{(1+5 \gamma)\frac{m - (d - 1)k}{2}}{k} \leq 2^{30\gamma \log (1/\gamma) k}\binom{\lambda(d)}{k - d +1},
  \end{equation*}
  where in the last inequality we also used $\sigma^2 \leq k^{2\alpha} \leq k^{1/2}$ (for $\alpha < 1/4$) and $\gamma \geq k^{-1/2}$.
\end{proof}

\section{A geometric observation in the proof of \Freiman's lemma}\label{sec:freiman-trick}
The goal of this appendix is to formalise the following simple geometric observation, first used in the proof of \Freiman's lemma.
\begin{prop}\label{prop:freiman-trick}
  Let \(d,r \in \N\) and let \(Z \subset \R^{d}\) be a finite set with $\rank(Z)=r$.
  Let $v \in Z$ be a vertex in the convex hull of $Z$ and assume $\rank(Z')=r$, where $Z'=Z \setminus\{v\}$.
  Then there exists a set $U = \{u_1, \ldots, u_{r+1}\}\subset Z$ such that $v+u_i \neq v+u_j$ for each $i \neq j$ and $(v+U) \cap (Z'+Z') = \emptyset$.
\end{prop}

This is a simple consequence of Observation 6.2 in \citep{campos2024independence}, which we state and prove below for convenience.

\begin{obs}\label{obs:facet-intersection}
  Let \(d,r \in \N\) and let \(Z \subset \R^{d}\) be a finite set.
  Let $v \in Z$ be a vertex in the convex hull of $Z$ and define $Z'=Z \setminus\{v\}$.
  Then there exists a hyperplane $\cH$ such that $v$ and $Z' \setminus \cH$ are on different sides of $\cH$.
  Moreover, $|\cH\cap Z'|\geq \rank(Z')$.
\end{obs}

\begin{proof}
  Write $\conv(Z') = \bigcap_{\cH \in \mathbf{H}} \cH^-$, where $\mathbf{H}$ is the collection of hyperplanes supporting the facets of $\conv(Z')$ and $\cH^-$ denotes a closed half-space defined by $\cH$.
  Since $v \not \in \conv(Z')$, there exists $\cH \in \mathbf{H}$ such that $v \not \in \cH^-$, so $v$ and $Z' \setminus \cH$ are on different sides of $\cH$.
  The second part of the statement follows from $\cH$ intersecting a facet of $\conv(Z')$ and the fact that every facet contains at least $\rank(Z')$ vertices, since it is a set with rank equal to $\rank(Z')-1$.
\end{proof}

Now we can prove \Cref{prop:freiman-trick}.
\begin{proof}[Proof of \Cref{prop:freiman-trick}]
  Let $\cH$ be the hyperplane given by \Cref{obs:facet-intersection} for $Z$ and $v$, and take $\oN{v} = \big(\cH \cap Z'\big) \cup \{v\}$.
  We claim any subset of size $r+1$ of $\oN{v}$ is a suitable choice for $U$.

  Observe first that, in fact, we can make such choice, since
  \[|\oN{v}| = |\cH \cap Z'| + 1 \ge \rank(Z') + 1 = r+1.\]
  To see that $\oN{v} + v$ and $Z' + Z'$ are disjoint, first notice that $2v \not \in Z' + Z'$ follows from $v$ being a vertex of $\conv(Z)$.
  For the remaining elements of $\oN{v}$, \ie $v' \in \cH \cap Z'$, $v' + v$ is not in $Z' + Z'$ because $(v' + v)/2$ is a midpoint of the segment connecting $v$ and $v'$, and this midpoint clearly lies outside $\conv(Z')$ by our choice of $\cH$.
\end{proof}
\end{document}